\documentclass[11pt,a4paper]{article}

\usepackage[naustrian, english]{babel} 
\usepackage[T1]{fontenc}
\usepackage{lmodern}
\usepackage{textcomp}
\usepackage{amssymb} 
\usepackage{mathtools}
\usepackage[pdftex]{graphicx}
\usepackage[export]{adjustbox}
\usepackage{caption}
\usepackage{color}
\usepackage{subfigure}
\usepackage[arrow, matrix, curve]{xy}
\usepackage[utf8]{inputenc}
\usepackage{url}
\usepackage{hyperref}
\usepackage{float}
\usepackage{chngcntr}
\usepackage{listings}
\usepackage{appendix}
\usepackage{pdfpages}
\usepackage{algorithm}
\usepackage{algorithmic}
\usepackage{geometry}
\usepackage{bbm}
\usepackage{bbold}
\usepackage{amsthm}
\usepackage[shortlabels]{enumitem}
\usepackage{mathtools}
\usepackage{fancyref}

\definecolor{codegreen}{rgb}{0,0.6,0}
\definecolor{codegray}{rgb}{0.5,0.5,0.5}
\definecolor{codepurple}{rgb}{0.58,0,0.82}
 
\lstdefinestyle{mystyle}{
		basicstyle=\tiny,
    commentstyle=\color{codegreen},
    keywordstyle=\color{blue},
    numberstyle=\tiny\color{codegray},
    stringstyle=\color{codepurple},
    breakatwhitespace=false,         
    breaklines=true,                 
    captionpos=b,                    
    keepspaces=true,                 
    numbers=left,                    
    numbersep=5pt,                  
    showspaces=false,                
    showstringspaces=false,
    showtabs=false,                  
    tabsize=2,
}
 
\newcommand{\R}{\mathbb{R}}
\newcommand{\N}{\mathbb{N}}
\newcommand{\pa}{\partial}

\newcommand{\ve}{\varepsilon}

\newcommand{\vp}{\varphi}

\newcommand{\md}{\mathrm{d}}

\newcommand{\h}{\mathcal{H}}

\newcommand{\prox}{\operatorname{prox}}
\newcommand{\dom}{\operatorname{dom}}
\newcommand{\dtx}{\dot{x}}
\newcommand{\dty}{\dot{y}}
\newcommand{\dtz}{\dot{z}}
\newcommand{\crit}{\operatorname{crit}}
\newcommand{\argmin}{\operatorname{argmin}}
\newcommand{\dist}{\operatorname{dist}}

\newtheorem{theorem}{Theorem}
\newtheorem{lemma}[theorem]{Lemma}
\newtheorem{cor}[theorem]{Corollary}
\theoremstyle{definition}
\newtheorem{Def}{Definition}  
\newtheorem{remark}{Remark}   
\newtheorem{example}{Example}
\newtheorem{ass}{Assumption}

\usepackage{tikz}
\usetikzlibrary{positioning}
\tikzset{>=stealth}
 
\begin{document}

\title{A forward-backward dynamical approach for nonsmooth problems with block structure coupled by a smooth function}
\author{Radu Ioan Bo\c{t}\thanks{University of Vienna, Faculty of Mathematics, Oskar-Morgenstern-Platz 1, 1090 Vienna, Austria, email: {\tt radu.bot@univie.ac.at}. Research partially supported by the Austrian Science Fund (FWF), project number I 2419-N32.} \and 
	Laura Kanzler\thanks{University of Vienna, Faculty of Mathematics, Oskar-Morgenstern-Platz 1, 1090 Vienna, Austria, email: {\tt laura.kanzler@univie.ac.at}. Research supported by the SFB F65 ``Taming Complexity in Partial Differential Systems'' 
and by the Doctoral Program ``Dissipation and Dispersion in Nonlinear PDEs'', project number W1245, which are funded by the Austrian Science Fund (FWF).}}
\maketitle

\begin{abstract}
	
In this paper we aim to minimize the sum of two nonsmooth (possibly also nonconvex) functions in separate variables connected by a smooth coupling function. To tackle this problem we chose a continuous forward-backward approach 
and introduce a dynamical system which is formulated by means of the partial gradients of the smooth coupling function and the proximal point operator of the two nonsmooth functions. Moreover, we consider variable rates of implicitness
of the resulting system. We discuss the existence and uniqueness of a solution and carry out the asymptotic analysis of its convergence  behaviour to a critical point of the optimization problem, when a regularization of the objective function fulfills 
the Kurdyka-\L ojasiewicz property.  We further provide convergence rates for the solution trajectory in terms of the \L ojasiewicz exponent. We conclude this work with numerical simulations which confirm and validate the analytical results.
\end{abstract}

 \vspace{1ex}

\begin{subjectclassification}
34G25, 37N40, 49J52, 90C26, 90C56
\end{subjectclassification} 

\section{Introduction}
We consider a block-structured optimization problem of the form
\begin{equation}\label{p}
\min_{(x,y) \in \R^n \times \R^m} \Psi(x,y) := f(x)+g(y)+H(x,y),
\end{equation}
where the following general assumptions on the functions $f$, $g$ and $H$ are made: 
\begin{ass}\label{A1}
\begin{itemize}
		\item $f: \R^n \to \overline{\R}$ and $g: \R^m \to \overline{\R}$ are proper and lower semicontinuous functions, where we denote $\overline{\R}:=\R \cup \{+\infty\}$, with $\inf_{\R^n}{f}> -\infty$ and  $\inf_{\R^m}{g}> -\infty$;
		\item $H \in C^1( \R^n \times \R^m)$, $\nabla H$ is Lipschitz continuous with constant $L$, and $\inf_{\R^n \times \R^m} H > -\infty$.
	\end{itemize}
\end{ass} 
We approach solving the optimization problem (\ref{p}) by associating to it the dynamical system
\begin{equation}\label{dynpalm}
\begin{cases}
\dtx(t)+x(t) \in \prox_{\frac{1}{\gamma_1 L}f}\left(x(t)-\frac{1}{\gamma_1 L}\nabla_xH(x(t), (1-\mu)(\dty(t)+y(t))+\mu y(t))\right)\\
\dty(t)+y(t)\in\prox_{\frac{1}{\gamma_2 L}g}\left(y(t)-\frac{1}{\gamma_2 L}\nabla_yH((1-\lambda)(\dtx(t)+x(t))+\lambda x(t)), y(t))\right) \\
(x(0),y(0))=(x_0,y_0),
\end{cases}
\end{equation}
where $\lambda, \mu \in [0,1]$ and the proximal map $\prox_{\alpha h}: \R^d \to {\cal P}(\R^d)$ for a proper and lower semicontinuous function $h: \R^d \to \bar{\R}$ and $\alpha>0$ is defined as (see \cite[Definition 1.22]{rocky})
$$\prox_{\alpha h}(x) := \argmin_{\substack{y \in \R^d}}\left\{h(y)+\frac{1}{2\alpha} \|y-x\|^2\right\}.$$
We chose the stepsizes $\frac{1}{\gamma_1 L}$ and $\frac{1}{\gamma_2 L}$ in dependence of the Lipschitz constant of the gradient of $H$ for some suitable $\gamma_1, \gamma_2 >0$. 

In \cite{bol}, Bolte, Sabach and Teboulle established the PALM (proximal alternating linearized minimization) algorithm for minimizing optimization problems of the form (\ref{p}). This work served as key motivation to study \eqref{dynpalm}, since an explicit time discretization for the case $\mu=1$ and $\lambda=0$ provides a direct correspondence between the ODE system and the PALM algorithm. In this special case the system (\ref{dynpalm}) is of the form
\begin{equation}\label{dynpalmspecial}
\begin{cases}
\dtx(t)+x(t) \in \prox_{\frac{1}{\gamma_1 L}f}\left(x(t)-\frac{1}{\gamma_1 L}\nabla_xH(x(t), y(t))\right) \\
\dty(t)+y(t)\in\prox_{\frac{1}{\gamma_2 L}g}\left(y(t)-\frac{1}{\gamma_2 L}\nabla_yH(\dtx(t)+x(t), y(t))\right) \\
(x(0),y(0))=(x_0,y_0).
\end{cases}
\end{equation}
Explicit time discretization of the first component of \eqref{dynpalmspecial} has the form
\begin{align}
x^{k+1} \in \prox_{\frac{1}{\gamma_1 L}f}\left(x^k-\frac{1}{\gamma_1 L}\nabla_xH(x^k, y^k)\right) \quad \forall k \geq 0,
\end{align}
where the time stepsize was set equal to one. On the other hand, explicit time discretization of the second component with stepsize one gives
\begin{align}
y^{k+1} \in \prox_{\frac{1}{\gamma_2 L}g}\left(y^k-\frac{1}{\gamma_2 L}\nabla_yH(x^{k+1}, y^k)\right) \quad \forall k \geq 0.
\end{align}
Therefore, we obtain the PALM algorithm for starting point $(x^0,y^0):=(x_0,y_0)$ from the ODE system (\ref{dynpalmspecial}). In \cite{bol} it was proved that the limit points of $(x^k,y^k)_{k \geq 0}$ are critical points of the optimization problem \eqref{p}. In addition, if the objective function satisfies the Kurdyka-\L ojasiewicz property and $(x^k,y^k)_{k \geq 0}$ is bounded, then the sequence converges to a critical point of \eqref{p}.

In case $\mu = \lambda =1$, explicit time discretization with stepsize one gives
$$x^{k+1} \in \prox_{\frac{1}{\gamma_1 L}f}\left(x^k-\frac{1}{\gamma_1 L}\nabla_xH(x^k, y^k)\right), \quad  y^{k+1} \in \prox_{\frac{1}{\gamma_2 L}g}\left(y^k-\frac{1}{\gamma_2 L}\nabla_yH(x^k, y^k)\right) \quad \forall k \geq 0,$$
which is the the preconditioned forward-backward algorithm for solving \eqref{p}. The forward-backward algorithm has been investigated in the fully nonconvex setting in \cite{att3, bot3}.

The approach of convex optimization problems and monotone inclusions from the perspective of dynamical systems has a long tradition, starting with the contributions of Crandall and Pazy \cite{cran} as well as Baillon, Brézis \cite{bai}, \cite{brez}, and Bruck \cite{bru}. 
In these works dynamical systems formulated as monotone inclusions and governed by subdifferential operators or, more general, maximal monotone operators in Hilbert spaces are investigated. 
The question of existence and uniqueness of trajectories generated by these dynamical systems is usually investigated in the framework of the Cauchy-Lipschitz Theorem, while the asymptotic convergence behavior of solutions to a minimizer of the convex optimization problem or a zero point 
of the governing maximal monotone operator builds on Lyapunov analysis. 

In the last years, dynamical systems of implicit type approaching monotone inclusions/convex optimization problems, namely, defined by means of resolvent/proximal operators have enjoyed much attention. Here, the pioneering works are by  by Antipin \cite{ant}, 
Bolte \cite{bol1}, and Abbas, Attouch and Svaiter \cite{abb1, abb, att2}. A general approach for addressing implicit dynamical systems was considered in \cite{bot1}, a dynamical system of forward-backward-forward type was matter of investigation in \cite{ban}, 
one of Douglas-Rachford type in \cite{cse}, while a primal-dual dynamical system approaching structured convex minimization problems was recently introduced in \cite{bot2}.

In what concerns nonconvex optimization problems, we want to mention that the problem of minimizing a general smooth function has been approached form the perspective of first-order and second-order gradient type dynamical systems in Simon \cite{sim}, Haraux and Jendoubi \cite{haraux2}, and Alvarez, 
Attouch, Bolte and Redont \cite{al}. In addition, proximal-gradient type dynamical systems of first-order and second-order have been investigated in \cite{bot} and \cite{bot4}, respectively, in relation to the minimization of the sum of a proper, convex and lower semicontinuous function and a general smooth function. 

The outline of this paper is the following: after some preliminaries, we will address in Section \ref{sec3} the question of existence and uniqueness of strong global solutions to \eqref{dynpalm}, 
which will be provided for the case $\mu=1$ and $f$ and $g$ are convex functions and in the framework of the global Picard-Lindel\"of Theorem. Section \eqref{sec4} is dedicated to the asymptotic analysis of the trajectories of the general system \eqref{dynpalm} towards a critical point of the objective 
function of \eqref{p}, expressed as a zero of the limiting subdifferential. To this end we will use three main ingredients (see \cite{att3, bol} for a similar approach in discrete and \cite{al, bot} in continuous time): (1) we will prove sufficient decrease of a regularized objective function along the trajectories; (2) we will show existence of a subgradient lower bound of the solution trajectories; (3) we will obtain convergence of the solution by taking use of the Kurdyka-\L ojasiewicz (KL) property of the objective function. Functions satisfying the KL property build a large class of functions and include semi-algebraic functions and functions having analytic features. We close our investigations by establishing convergence rates for the trajectories expressed in terms of the \L ojasiewicz exponent of the regularized objective function. We will conclude by confirming and validating the analytical results through numerical simulations.

\section{Preliminaries} 

For $d \geq 1$, we consider on $\R^d$ the Euclidean scalar product and the induced norm. These are denoted by $\langle \cdot, \cdot \rangle$ and $\| \cdot \|$, respectively, independently from the value of $d$, since confusion is not possible.

Let $h: \R^d \to \overline{\R}$ be a given function. Its \textit{effective domain} is defined as $\dom{h}:=\{x \in \R^d: \; h(x) < +\infty\}$ and we say that $h$ is \textit{proper}, if $\dom h \neq \emptyset$. 

\begin{Def}
	\begin{itemize}
		\item Let $h: \: \R^d \to \overline{\R}$ be proper and lower semicontinuous. The \textit{Fr\'{e}chet subdifferential} of $h$ at $x \in \dom h$ is defined as 
		$$\pa_F h(x) := \left\{\xi \in \R^d: \; \underset{z \to x}{\lim \inf} \frac{h(z)-h(x)- \langle \xi, z-x \rangle}{\|z-x\|} \geq 0 \right\}.$$
		For $x \notin  \dom h$, we set by convention $\pa_F h(x) := \emptyset$. 
		\item The \textit{limiting (Mordukhovich) subdifferential} of $h$ at $x \in \dom h$ (see \cite{mor}) is then defined as the sequential closure of $\pa_F h(x)$ in the following way
		$$\pa h(x) := \left\{\xi \in \R^d; \; \exists x_k \to x, \; h(x_k) \to h(x) \text{ and } \pa_F h(x_k)\ni \xi_k \to \xi \text{, as } k \to \infty \right\}.$$
For $x \notin  \dom h$, we set by convention $\pa h(x) := \emptyset$.  We denote by $\dom \pa h := \{x \in \R^d: \partial h(x) \neq \emptyset\}$ the \textit{domain} of the limiting subdifferential $\pa h$.
	\end{itemize}
\end{Def}

\begin{remark}
	\begin{enumerate}[i)]
		\item From the definition it follows that for all $x \in \R^d$ it holds $\pa_F h(x) \subseteq \pa h(x)$. While $\pa h(x)$ is closed, $\pa_F h(x)$ is convex and closed (\cite[Theorem 8.6]{rocky}). 
		\item With $\crit (h) := \{x \in \R^d: \; 0 \in \pa h(x)\}$ we denote the set of \textit{(limiting-)critical points of $h$}. Also in our nonsmooth setting \textit{Fermat's theorem} holds, i.e. if $x \in \R^d$ is a local minimizer of $h$, then $0 \in \pa h(x)$, i.e. $x \in \crit (h)$.
		\item Should $h$ be continuously differentiable at $x \in \R^n$, then we have $\pa h(x) = \{\nabla h(x)\}$.
		\item For the sum of a proper and lower semicontinuous function $h:\R^d \to \overline{\R}$  and a  continuously differentiable function $k : \R^d \rightarrow \R$ it holds $\pa (h+k)(x) = \pa h(x) +\nabla k(x)$ for every $x \in \R^d$ (\cite{rocky}). 
	\end{enumerate}
\end{remark}

\begin{remark}
	\begin{enumerate}[i)]
		\item The proximal operator of a nonconvex function is in general a set-valued map. For a proper and lower semicontinuous function $h:\R^d \to \overline{\R}$ with $\inf h > -\infty$ we have that for every $\alpha > 0$ and every $x \in \R^n$ the set $\prox_{\alpha f}(x)$ is nonempty and compact.
		\item In our considerations below, we will use the following property connecting the \textit{proximal map} with the \emph{limiting subdifferential}. For every proper and lower semicontinuous function $h: \R^d \to \overline{\R}$ and constant $\alpha>0$ it holds
		\begin{align} \label{prox1}
		p \in \prox_{\alpha h}(x) \Rightarrow \frac{1}{\alpha} (x - p) \in \pa h(p) \quad \forall x \in \R^d.
		\end{align}
	\end{enumerate}
\end{remark}

Since we will have to deal with locally absolutely continuous solution trajectories of (\ref{dynpalm}), we recall the following definition.

\begin{Def}
	A function $x:[0,+\infty) \to \R^n$ is called \emph{locally absolutely continuous} if $x:[0,T] \to \R^n$  is absolutely continuous  for all $T >0$, namely, there exists an integrable function $y: [0,T] \to \R^n$ such that
		$$x(t)=x(0) + \int_0^t y(s) \;\md s \quad \forall t \in [0,T].$$
\end{Def}
We point out two important properties of absolute continuous trajectories, which will be of much use later on.
\begin{remark}\label{ac}
	\begin{enumerate}[i)]
		\item An absolutely continuous function is almost everywhere (a.e.) differentiable, its derivative coincides a.e. with its distribution derivative, and by the above integration formula it is possible to recover the function from its derivative.
		\item For an absolutely continuous function $x: [0,T] \to \R^d$, where $T>0$, and a Lipschitz continuous function $K$ with Lipschitz constant $\beta \geq 0$, the composition $z:=K \circ x$ is also absolutely continuous.  
		Furthermore, $z$ is differentiable a.e. in $[0,T]$ and the inequality $\|\dtz(t)\| \leq \beta \|\dtx(t)\|$ holds for a.e. $t \in [0,T]$. 
	\end{enumerate}
\end{remark}

Next, we introduce the notion of solution we will mostly deal with.  
\begin{Def}\label{strsol}
	We say that $z:=(x,y): \; [0,+\infty) \to \R^n \times \R^m$ is a \emph{strong global solution} of (\ref{dynpalm}), if it satisfies the following properties:
	\begin{enumerate}[a)]
		\item the functions $x$ and $y$ are locally absolutely continuous;
		\item $$\dtx(t)+x(t) \in \prox_{\frac{1}{\gamma_1 L}f}\left(x(t)-\frac{1}{\gamma_1 L}\nabla_xH(x(t), (1-\mu)(\dty(t)+y(t))+\mu y(t))\right)$$ and $$\dty(t)+y(t)\in\prox_{\frac{1}{\gamma_2 L}g}\left(y(t)-\frac{1}{\gamma_2 L}\nabla_yH((1-\lambda)(\dtx(t)+x(t))+\lambda x(t)), y(t))\right)$$ 
		for a.e. $t \in [0,+\infty)$;
\item the functions $t \mapsto f(\dtx(t) + x(t))$ and $t \mapsto g(\dty(t) + y(t))$ are locally absolutely continuous;
		\item $(x(0),y(0)) = (x_0,y_0)$.
	\end{enumerate}
\end{Def}

Also the following two results for locally absolutely continuous functions will be crucial for the asymptotic analysis of the trajectories of (\ref{dynpalm}) (see for example \cite{abb}).

\begin{lemma}\label{fejer1}
	Suppose that $F: [0,+\infty) \to \R$ is locally absolutely continuous and bounded from below. Furthermore, assume that $\exists \; G \in L^1([0,+\infty);\R)$ such that for a.e. $t \in [0,+\infty)$
	$$\frac{\md}{\md t} F(t) \leq G(t).$$
	Then $\exists \lim_{t \to +\infty} F(t) \in \R$.
\end{lemma} 

\begin{lemma}\label{fejer2}
	If $1 \leq p<\infty$, $1\leq r \leq \infty$, $F: [0,+\infty) \to [0,+\infty)$ is locally absolutely continuous, $F \in L^p([0,+\infty);\R)$, $G \in L^r([0,+\infty);\R)$ and for a.e. $t \in [0,+\infty)$ 
	$$\frac{\md}{\md t} F(t)\leq G(t)$$
	holds, then $\lim_{t \to +\infty}F(t)=0$.
\end{lemma}

\section{Existence and uniqueness of a trajectory}\label{sec3}

The aim of this section is to provide a setting in which the existence and uniqueness of a trajectory of the the dynamical system (\ref{dynpalm}) is guaranteed.  To this end:
\begin{itemize}
	\item we ask the nonsmooth functions $f$ and $g$ to be in addition \emph{convex}. We recall that the proximal operator of a proper, convex and lower semicontinuous function is single-valued and non-expansive (see, e.g. \cite[Proposition 12.28]{bc}).
	\item we choose the parameter $\mu = 1$, while $\lambda \in [0,1]$ remains general. 
\end{itemize}
The above setting allows us to rewrite the dynamical system (\ref{dynpalm}) in an explicit form  as
\begin{equation}\label{dynpalmmu1}
\begin{cases}
\dtx(t)+x(t)=\prox_{\frac{1}{\gamma_1 L}f}\left(x(t)-\frac{1}{\gamma_1 L}\nabla_xH(x(t), y(t))\right) \\
\dty(t)+y(t)=\prox_{\frac{1}{\gamma_2 L}g}\left(y(t)-\frac{1}{\gamma_2 L}\nabla_yH((1-\lambda)(\dtx(t)+x(t))+\lambda x(t)), y(t))\right)\\
(x(0),y(0))=(x_0,y_0).
\end{cases}
\end{equation}
In the following, we will show that the operator $\Gamma: \R^n \times \R^m \to  \R^n \times \R^m$, $\Gamma(x,y):=(u,v)$, where 
$$\begin{pmatrix} u \\ v \end{pmatrix} = \begin{pmatrix}
	\prox_{\frac{1}{\gamma_1 L}f}\left(x-\frac{1}{\gamma_1 L}\nabla_xH(x,y)\right) -x \\
	\prox_{\frac{1}{\gamma_2 L}g}\left(y-\frac{1}{\gamma_2 L}\nabla_yH((1-\lambda)(u+x)+\lambda x,y)\right) -y
\end{pmatrix},$$
is Lipschitz continuous, from which we will conclude existence and uniqueness of a solution to (\ref{dynpalmmu1}) by applying the global version of the Picard-Lindelöf Theorem (see, e.g.  (\cite[Theorem 2.2]{teschl}). 

\begin{theorem}\label{existence}
	Let $f$, $g$ and $H$ fulfill Assumption (\ref{A1}) and let further $f$ and $g$ be convex. Then for every starting point $(x_0,y_0) \in \R^n \times  \R^m$ the dynamical system (\ref{dynpalm}) for $\mu=1$ and $\lambda \in [0,1]$ has a unique strong global solution $z=(x,y) : [0,+\infty) \rightarrow \R^n \times  \R^m$,  
	which is in addition continuously differentiable.	
\end{theorem}

\begin{proof}
Let $(x,y), (\tilde{x},\tilde{y}) \in \R^n \times  \R^m$. We denote $(u,v):=\Gamma(x,y)$, $(\tilde{u},\tilde{v}):=\Gamma(\tilde{x},\tilde{y})$ and estimate 
\begin{align*}
 & \ \left\| u - \tilde{u} \right\|^2 = \left\| \prox_{\frac{1}{\gamma_1 L}f}\left(x-\frac{1}{\gamma_1 L}\nabla_xH(x,y)\right) -x -\prox_{\frac{1}{\gamma_1 L}f}\left(\tilde{x}-\frac{1}{\gamma_1 L}\nabla_xH(\tilde{x},\tilde{y})\right) + \tilde{x} \right\|^2 \\
 \leq & \ 2\left\| x-\tilde{x} \right\|^2 + 2\left\| x-\frac{1}{\gamma_1 L}\nabla_xH(x,y) +\frac{1}{\gamma_1 L}\nabla_xH(\tilde{x},\tilde{y}) -\tilde{x}\right\|^2\\
 \leq & \ 6 \left\| x-\tilde{x} \right\|^2 + \frac{4}{\gamma_1^2 L^2} \left\| \nabla_xH(x,y) - \nabla_xH(\tilde{x},\tilde{y}) \right\|^2 \\
 \leq & \ 6 \left\| x-\tilde{x} \right\|^2 + \frac{4}{\gamma_1^2 } \left(\left\| x-\tilde{x} \right\|^2 + \left\|y-\tilde{y}\right\|^2 \right) \\
 = & \ \left(6+\frac{4}{\gamma_1^2}\right) \left\| x-\tilde{x} \right\|^2 +\frac{4}{\gamma_1^2 } \left\|y-\tilde{y}\right\|^2,
\end{align*}
where we first used the \textit{non-expansiveness} (i.e. 1-Lipschitz continuity) of the proximal operator, while the last inequality is due to the $L$-Lipschitz continuity of $\nabla H$. Similarly we obtain 
\begin{align*}
& \ \left\| v - \tilde{v} \right\|^2 \\
\leq & \ 2\left\| y-\tilde{y} \right\|^2 + 2\left\| y-\frac{1}{\gamma_2 L}\nabla_yH((1-\lambda)(u+x)+\lambda x,y) +\frac{1}{\gamma_2 L}\nabla_yH((1-\lambda)(\tilde{u}+\tilde{x})+\lambda \tilde{x},\tilde{y}) +\tilde{y} \right\|^2\\
\leq & \ 6 \left\| y-\tilde{y} \right\|^2 + \frac{4}{\gamma_2^2 L^2} \left\| \nabla_yH((1-\lambda)(u+x)+\lambda x,y) - \nabla_yH((1-\lambda)(\tilde{u}+\tilde{x})+\lambda \tilde{x},\tilde{y}) \right\|^2 \\
 \leq & \ 6 \left\| y-\tilde{y} \right\|^2 + \frac{4}{\gamma_2^2 } \left(\|(1-\lambda) [(u+x) - (\tilde{u}+\tilde{x})]+\lambda(x-\tilde{x})\|^2 + \left\| y-\tilde{y} \right\|^2 \right) \\
 \leq & \ 6 \left\| y-\tilde{y} \right\|^2 + \frac{4}{\gamma_2^2} \left((1-\lambda)^2 \|u-\tilde{u}\|^2+\|x-\tilde{x}\|^2 + \left\| y-\tilde{y} \right\|^2 \right) \\
= & \ \frac{4}{\gamma_2^2 }(1-\lambda)^2 \|u-\tilde{u}\|^2+\frac{4}{\gamma_2^2}\|x-\tilde{x}\|^2 + \left(6+\frac{4}{\gamma_2^2 }\right) \left\| y-\tilde{y} \right\|^2.
\end{align*}
Putting the estimates for $\|u-\tilde{u}\|$ and $\|v - \tilde{v}\|$ together we obtain
\begin{align*}
&\|\Gamma(x,y) - \Gamma(\tilde{x},\tilde{y})\|^2 = \|u - \tilde{u}\|^2 + \|v - \tilde{v}\|^2 \\
\leq & \ \left[6 +\frac{4}{\gamma_1^2} + \frac{4}{\gamma_2^2} + \frac{16(1-\lambda)^2}{\gamma_1^2 \gamma_2^2} +\frac{24(1-\lambda)^2}{\gamma_2^2}\right] \left\| x-\tilde{x} \right\|^2 +\left[6+\frac{4}{\gamma_1^2}+\frac{4}{\gamma_2^2} + \frac{16(1-\lambda)^2}{\gamma_1^2\gamma_2^2}\right] \left\|y-\tilde{y}\right\|^2  \\
\leq & \ \beta(\lambda, \gamma_1,\gamma_2)^2 \|(x,y) - (\tilde{x}, \tilde{y})\|^2,
\end{align*}
where
\begin{align*}
	\beta(\lambda,\gamma_1,\gamma_2) := \left(6 +\frac{4}{\gamma_1^2} + \frac{4+24(1-\lambda)^2}{\gamma_2^2} + \frac{16(1-\lambda)^2}{\gamma_1^2 \gamma_2^2} \right)^{\frac{1}{2}}.
\end{align*}
Therefore the operator $\Gamma$ is Lipschitz continuous and, according to the Picard-Lindel\"of Theorem, a unique continuously differentiable strong global solution $z=(x,y) : [0,+\infty) \rightarrow \R^n \times  \R^m$ exists. This means that the statements a), b), and d) in Definition \ref{strsol} are true.

This means that for a.e. $t \geq 0$ 
\begin{align}
-\gamma_1 L \dtx(t) -\nabla_xH(x(t), (1-\mu)(\dty(t)+y(t))+\mu y(t)) \in \pa f(\dtx(t)+x(t)), \label{paf} \\ 
-\gamma_2 L \dty(t) -\nabla_yH((1-\lambda)(\dtx(t)+x(t))+\lambda x(t), y(t)) \in \pa g(\dty(t)+y(t)), \label{pag}
\end{align}
thus $\dtx(t)+x(t) \in \dom f$ and  $\dty(t)+y(t) \in \dom g$. 

From Remark \ref{ac} ii) it follows that the first derivative $\dtz = (\dtx, \dty)$ is also locally absolutely continuous. This follows since the solution trajectory $z(t)=(x(t),y(t))$ is a $C^1$ function and therefore locally absolutely continuous and by definition of the dynamical system (\ref{dynpalmmu1}) 
it can be written as $(\dtx(t),\dty(t))= \Gamma \circ (x(t),y(t))$, where $\Gamma$ is $\beta(\lambda,\gamma_1,\gamma_2)$-Lipschitz continuous as we just showed. Moreover, the second derivatives of the partial trajectories exist for a.e. $t \in [0,+\infty)$ 
and the inequality $$\|(\ddot{x}(t),\ddot{y}(t))\| \leq \beta(\lambda,\gamma_1,\gamma_2) \|(\dtx(t),\dty(t))\|$$ holds for a.e. $t>0$. 
Let be $T >0$.  Then we have $\dot x + x, \ddot x + \dot x \in L^2([0,T], \R^n)$ and $ \dot y + y,  \ddot y + \dot y \in  L^2([0,T], \R^m)$. Using that $H$ is continuously differentiable, it yields that 
$-\gamma_1 L \dot x -\nabla_xH(x, (1-\mu)(\dot y+y)+\mu y) \in L^2([0,T], \R^n)$ and $-\gamma_2 L \dot y -\nabla_yH((1-\lambda)(\dot x + x)+ \lambda x, y) \in  L^2([0,T], \R^m)$. Thus, according to \cite[Lemma 3.3]{brez}, the functions $t \mapsto f(\dtx(t) + x(t))$ and $t \mapsto g(\dty(t) + y(t))$ are locally absolutely continuous. This concludes the proof.
\end{proof}

\section{Asymptotic analysis}\label{sec4}

\subsection{A preparatory result}
In order to prove convergence of trajectories of (\ref{dynpalm}), we follow the general approach that consists of several steps. In the first step, we will find a \textit{Lyapunov functional} with sufficient decrease of its time derivative, 
while in the second step we search for \textit{subgradient lower bounds} for the derivative of the trajectory. With standard arguments from \cite{haraux} one can show that the set of limit points of the trajectory is nonempty, compact and connected. 
Finally, global convergence of a trajectory of (\ref{dynpalm}) to a critical point will be established, by requiring that a regularization of the objective function fulfills the Kurdyka-\L ojasiewicz property.

We again have to require additional assumptions on the involved functions. Since the required additional assumptions on the involved functions  do not necessarily include, or are not as strong as the ones from the previous section, existence of a solution trajectory to (\ref{dynpalm}) itself will be presupposed. 
For performing asymptotic analysis it suffices to require less regularity of the solution to (\ref{dynpalm}) than we where able to obtain in Theorem \ref{existence} for the special system (\ref{dynpalmmu1}). More precisely, we just have to require the existence of a \textit{strong global solution} to the 
dynamical system as defined in Definition \ref{strsol}.

\begin{ass} \label{A4}
		 There exists a locally  absolutely continuous solution $z =(x,y)$ to the dynamical system (\ref{dynpalm}) with absolutely continuous derivative and an almost everywhere existing second derivative fulfilling the 
		 estimate $\|(\ddot{x}(t), \ddot{y}(t)\| \leq \beta(\lambda,\gamma_1,\gamma_2) \|(\dtx(t), \dty(t)\|$ for a.e. $t \in [0,+\infty)$ and some constant $\beta>0$, possibly depending on $\lambda$ and the stepsizes $\gamma_1,\gamma_2$.
\end{ass}

\begin{remark}
The statement of Assumption \ref{A4} is true for the case where $\mu=1$ and $f$ and $g$ are convex, as it was shown in the previous section. Nevertheless, as we will see as follows, the analysis of the asymptotic behavior of the trajectory can be carried out in a more general setting.
\end{remark}
 
The first ingredient of the asymptotic analysis will be a \textit{sufficient decrease} property of the derivative of the trajectory. For obtaining it, we need to control the time derivative of $f$ and $g$ along trajectories and therefore to assume that $f$ and $g$ admit a \emph{chain rule}.

\begin{Def}
Let $h: \R^d \to \overline{\R}$ be a proper and lower semicontinuous function. We say that $h$ admits the \emph{chain rule} if, for every locally absolutely continuous arc  $z: [0,+\infty) \to \R^d$ fulfilling $z(t) \in \dom \pa h$ for a.e. $t \geq 0$ and such that $t \mapsto h(z(t))$ is locally absolutely continuous, 
it holds for a.e. $t \geq 0$
\begin{align}\label{chain}
	\frac{\md}{\md t}h(z(t)) &= \langle \xi, \dtz(t) \rangle \ \mbox{for all} \ \xi \in \partial h(z(t)).
\end{align}
\end{Def}

\begin{remark}
\begin{itemize}
		\item If $h$ is, in addition, \textit{convex}, then it admits the chain rule (see \cite[Lemma 3.3]{brez}, \cite{att}).  Notice that if $z: [0,+\infty) \to \R^d$ is a locally absolutely continuous arc such that $z(t) \in \dom h$ and there exists $\xi(t) \in \partial h(z(t))$ for a.e. 
		$t \geq 0$, and for every $T >0$ it holds $z \in L^2([0,T], \R^d)$, $\dot z \in L^2([0,T], \R^d)$  and $\xi \in L^2([0,T], \R^d)$, then $t \mapsto h(z(t))$ is locally absolutely continuous.	

	\item A locally Lipschitz function $h : \R^d \rightarrow \R$ that is \textit{subdifferentially regular}
admits a chain rule (see \cite[Lemma 5.4]{dav}). We recall that a locally Lipschitz function $h: \R^d \to \R$ is called subdifferentially regular at a point $u \in \R^d$, if for all $\xi \in \pa h (u)$ the inequality
			$$h(v) \geq h(u) + \langle \xi, v -u \rangle + o(\|v-u\|)$$ 
			holds as $v \to u$.

Notice that in this case, if $z: [0,+\infty) \to \R^d$ is a locally absolutely continuous arc, then $t \mapsto h(z(t))$ is locally absolutely continuous, too, and the limiting subdifferential of $h$ is nothing else than \textit{Clarke's subdifferential} of $h$.

	\item A locally Lipschitz function $h : \R^d \rightarrow \R$, which is \emph{Whitney $C^1$-stratifiable}, admits a chain rule (see \cite[Theorem 5.8]{dav}). A function $h: \R^d \to \R$ is called Whitney $C^p$-stratifiable, if its graph admits a Whitney $C^p$-stratification. A Whitney $C^p$-stratification $\mathcal{A}$ of a set ${Q} \subseteq \R^d$ is a partition of ${Q}$ into finitely many nonempty \textit{$C^p$ smooth manifolds}, called \textit{strata}, satisfying the following two conditions:
		\begin{itemize}
			\item For any two strata $L, M \in \mathcal{A}$ it holds:
			$$L \cap \operatorname{cl}(M) \neq \emptyset \quad \Longrightarrow \quad L \subset \operatorname{cl}(M).$$
			\item For any sequence $(u_k)_{k \geq 0}$ in a stratum $M$ converging to a point $u$ in a stratum $L$, if the corresponding \emph{normal vectors} $v_k \in N_M(u_k)$ converge to a vector $v$, then $v \in N_L(u)$ holds. 
		\end{itemize}
A set $M \subseteq \R^d$ is a  \textit{$C^p$ smooth manifold} if there is an integer $r \in \N$ such that around any point $u \in M$, there is a neighborhood $U$ and a $C^p$-smooth map $F : U \rightarrow \R^{d-r}$ with $\nabla F(u)$ of full rank and satisfying $M \cap U = \{y \in U: F(y) =0\}$. Then the \textit{tangent space} $T_M(u)$ to $M$ at $u$ is the null space of $\nabla F(u)$, and the \textit{normal space} $N_M(u)$ to $M$ at $u$ is the orthogonal space to $T_M(u)$.

Functions having as graphs semialgebraic, subanalytic or even sets that are definable in an o-minimal structure are Withney $C^\infty$-stratifiable.
\end{itemize}
\end{remark}

\begin{theorem}\label{theorem1}
Let $f$, $g$ and $H$ fulfill Assumption \ref{A1} and $f$, $g$ admit a chain rule. Let further $z=(x,y) : [0,+\infty) \rightarrow \R^n \times \R^m$ be a strong global solution of the dynamical system (\ref{dynpalm}) fulfilling Assumption \ref{A4}, where the constants are chosen such that
	\begin{align}\label{g1} 
		\min\{\gamma_1, \gamma_2\} &> \beta(\lambda,\gamma_1,\gamma_2) \max{\bigg\{(1+\lambda +\mu + \lambda^2),(1+\lambda+\mu +\mu^2)\bigg\}} + \frac{\lambda+\mu}{2}. 
	\end{align}
Then the following statements are true:
	\begin{enumerate}
		\item $\dtz \in L^2([0,+\infty);\R^n \times \R^m])$ and $\lim_{t \to +\infty} \dtz(t) = 0$;
		\item $\exists \lim_{t \to +\infty} \Psi(\dtz(t)+z(t)) \in \R$.
	\end{enumerate}
\end{theorem}

\begin{proof}
Using the characterization (\ref{prox1}) of the proximal map in the system (\ref{dynpalm}), we obtain that \eqref{paf}-\eqref{pag} hold for a.e. $t \geq 0$. Since $f$ and $g$ admit a chain rule (\ref{chain}), it holds
\begin{align*}
&\frac{\md}{\md t} f(\dtx(t)+x(t))=- \Big\langle \ddot{x}(t)+\dtx(t), \, \gamma_1 L \dtx(t) +  \nabla_xH(x(t), (1-\mu)(\dty(t)+y(t))+\mu y(t)) \Big\rangle \\
&\frac{\md}{\md t} g(\dty(t)+y(t))=- \Big\langle \ddot{y}(t)+\dty(t), \, \gamma_2 L \dty(t) + \nabla_yH((1-\lambda)(\dtx(t)+x(t))+\lambda x(t), y(t)) \Big\rangle
\end{align*}
for a.e. $t \geq 0$.
Using the identity
\begin{align*}
\ &\frac{\md}{\md t}H(\dtx(t)+x(t),\dty(t) +y(t)) \\
= \ &\Big\langle \nabla_x H(\dtx(t)+x(t),\, \dty(t) +y(t)), \ddot{x}(t) + \dtx(t) \Big\rangle + \Big\langle \nabla_y H(\dtx(t)+x(t),\, \dty(t) +y(t)), \ddot{y}(t) + \dty(t) \Big\rangle, 
\end{align*}
we can compute
\begin{align*}
\ &\frac{\md}{\md t} \Psi(\dtx(t) + x(t), \dty(t)+y(t)) \\
= \ &  \frac{\md}{\md t}H(\dtx(t)+x(t),\dty(t) +y(t)) +  \frac{\md}{\md t} f(\dtx(t)+x(t)) + \frac{\md}{\md t} g(\dty(t)+y(t)) \\
= \ & \Big\langle \ddot{x}(t) + \dtx(t), \, -\gamma_1 L \dtx(t) - \nabla_xH(x(t), (1-\mu)(\dty(t)+y(t))+\mu y(t)) + \nabla_xH(\dtx(t)+ x(t), \dty(t)+y(t)) \Big\rangle \\
+ \ & \Big\langle \ddot{y}(t) + \dty(t), \, -\gamma_2 L \dty(t) - \nabla_yH((1-\lambda)(\dtx(t)+x(t))+\lambda x(t), y(t)) + \nabla_yH(\dtx(t)+ x(t), \dty(t)+y(t))\Big\rangle
\end{align*}
for a.e. $t \geq 0$. After first factorizing out the components of the inner product and then using the Cauchy-Schwarz inequality, we obtain
\begin{align*}
& \frac{\md}{\md t} \Psi(\dtx(t) + x(t), \dty(t)+y(t)) \notag\\
= & -\gamma_1 L \|\dtx(t)\|^2 -\gamma_2 L \|\dty(t)\|^2 -\gamma_1 L \Big\langle \ddot{x}(t), \, \dtx(t) \Big\rangle - \gamma_2 L \Big\langle \ddot{y}(t), \, \dty(t) \Big\rangle \notag \\
&+ \Big\langle \ddot{x}(t) + \dtx(t),\,  \nabla_xH(\dtx(t)+ x(t), \dty(t)+y(t)) - \nabla_xH(x(t), \, (1-\mu)(\dty(t)+y(t))+\mu y(t)) \Big\rangle \notag \\
& +  \Big\langle \ddot{y}(t) + \dty(t), \, \nabla_yH(\dtx(t)+ x(t), \dty(t)+y(t))- \nabla_yH((1-\lambda)(\dtx(t)+ x(t)) + \lambda x(t), \, y(t)) \Big\rangle \notag \\
\leq &  -\gamma_1 L \|\dtx(t)\|^2 -\gamma_2 L \|\dty(t)\|^2 - \frac{\gamma_1 L}{2} \frac{\md}{\md t} \|\dtx(t)\|^2 - \frac{\gamma_2 L}{2} \frac{\md}{\md t} \|\dty(t)\|^2 \\
&+ \|\ddot{x}(t) + \dtx(t)\| \| \nabla_xH(\dtx(t)+ x(t), \dty(t)+y(t)) - \nabla_xH(x(t), \, (1-\mu)(\dty(t)+y(t))+\mu y(t))\| \notag \\
&+ \| \ddot{y}(t) + \dty(t)\| \|  \nabla_yH(\dtx(t)+ x(t), \dty(t)+y(t))- \nabla_yH((1-\lambda)(\dtx(t)+ x(t)) + \lambda x(t), \, y(t)) \| \notag 
\end{align*}
for a.e. $t \geq 0$. In the next step we apply the Lipschitz continuity of the gradient of $H$, which yields
\begin{align*}
 & \frac{\md}{\md t} \Psi(\dtx(t) + x(t), \dty(t)+y(t))\notag\\
\leq & -\gamma_1 L \|\dtx(t)\|^2 -\gamma_2 L \|\dty(t)\|^2 - \frac{\gamma_1 L}{2} \frac{\md}{\md t} \|\dtx(t)\|^2 - \frac{\gamma_2 L}{2} \frac{\md}{\md t} \|\dty(t)\|^2 \notag \\
&+L (\|\dtx(t)\| + \mu \|\dty(t)\|) \|\ddot{x}(t) + \dtx(t)\| +L (\lambda \|\dtx(t)\| + \|\dty(t)\|) \|\ddot{y}(t) + \dty(t)\| \\
\leq & -\gamma_1 L \|\dtx(t)\|^2 -\gamma_2 L \|\dty(t)\|^2 - \frac{\gamma_1 L}{2} \frac{\md}{\md t} \|\dtx(t)\|^2 - \frac{\gamma_2 L}{2} \frac{\md}{\md t} \|\dty(t)\|^2 \notag\\
& +L (\|\dtx(t)\| + \mu \|\dty(t)\|) \|\ddot{x}(t)\| + L (\lambda \|\dtx(t)\| + \|\dty(t)\|) \|\ddot{y}(t)\| \notag \\
&+ L(\|\dtx(t)\| + \mu \|\dty(t)\|) \|\dtx(t)\| + L (\lambda \|\dtx(t)\| + \|\dty(t)\|) \|\dty(t)\|\notag
\end{align*}
for a.e. $t \geq 0$. Further, we have
\begin{align}\label{est3}
 &L (\|\dtx(t)\| + \mu \|\dty(t)\|) \|\ddot{x}(t)\| + L (\lambda \|\dtx(t)\| + \|\dty(t)\|) \|\ddot{y}(t)\| \notag \\
\leq \ & L \sqrt{\|\ddot{x}(t)\|^2+\|\ddot{y}(t)\|^2} \sqrt{(\|\dtx(t)\|+\mu\|\dty(t)\|)^2+(\lambda\|\dtx(t)\|+\|\dty(t)\|)^2} \\
= \ & L \|(\ddot{x}(t),\ddot{y}(t))\| \sqrt{(1+\lambda^2)\|\dot{x}(t)\|^2 + (1+\mu^2)\|\dot{y}(t)\|^2 + 2(\lambda+\mu)\|\dot{x}(t)\|\|\dot{y}(t)\|} \notag \\
\leq \ & L \beta(\lambda,\gamma_1,\gamma_2) \|(\dtx(t), \dty(t))\| \sqrt{(1+\lambda +\mu + \lambda^2)\|\dtx(t)\|^2+(1+\lambda+\mu +\mu^2)\|\dty(t)\|^2}, \notag
\end{align}
for a.e. $t \geq 0$, where in the last inequality we used Assumption \ref{A4}. By defining 
\begin{align*}
M_{\lambda,\mu}:=\max{\bigg\{(1+\lambda +\mu + \lambda^2)^\frac{1}{2},(1+\lambda+\mu +\mu^2)^\frac{1}{2}\bigg\}},
\end{align*}
we can write
\begin{align*}
L \Big(\|\dtx(t)\| + \mu \|\dty(t)\|\Big) \|\ddot{x}(t)\| + L \Big(\lambda \|\dtx(t)\| + \|\dty(t)\|\Big) \|\ddot{y}(t)\|  \leq L \beta(\lambda,\gamma_1,\gamma_2) M_{\lambda,\mu}\Big(\|\dtx(t)\|^2 + \|\dty(t)\|^2\Big)
\end{align*}
for a.e. $t \geq 0$. On the other hand, it holds
\begin{align*}
& L\Big(\|\dtx(t)\| + \mu \|\dty(t)\|\Big) \|\dtx(t)\| + L \Big(\lambda \|\dtx(t)\| + \|\dty(t)\|\Big) \|\dty(t)\| \\
= \ & L\Big(\|\dtx(t)\|^2+\|\dty(t)\|^2\Big) + L(\lambda+\mu)\|\dtx(t)\|\|\dty(t)\| \\
\leq \ & L\left(1+\frac{\lambda+\mu}{2}\right) \Big(\|\dtx(t)\|^2+\|\dty(t)\|^2\big)
\end{align*}
for a.e. $t \geq 0$. Combining the estimates above we obtain for a.e.  $t \geq 0$ the following inequality
\begin{align*}
 & \frac{\md}{\md t} \Psi(\dtx(t) + x(t), \dty(t)+y(t))\\ 
\leq & - \frac{\gamma_1 L}{2} \frac{\md}{\md t} \|\dtx(t)\|^2 - \frac{\gamma_2 L}{2} \frac{\md}{\md t} \|\dty(t)\|^2 \\
 &-L \left(\gamma_1 - \beta(\lambda,\gamma_1,\gamma_2) M_{\lambda,\mu} - \frac{\lambda+\mu}{2}\right) \|\dtx(t)\|^2 - L \left(\gamma_2-\beta(\lambda,\gamma_1,\gamma_2) M_{\lambda,\mu} -\frac{\lambda+\mu}{2} \right) \|\dty(t)\|^2
\end{align*}
or, equivalently,
\begin{align}\label{hdec}
& \frac{\md}{\md t} \left( \Psi(\dtx(t) + x(t), \dty(t)+y(t)) + \frac{1}{2}\left\|\left(\sqrt{\gamma_1 L} \dtx(t), \,\sqrt{\gamma_2 L} \dty(t)\right)\right\|^2 \right) \notag \\
\leq & -L \left(\gamma_1 - \beta(\lambda,\gamma_1,\gamma_2) M_{\lambda,\mu} - \frac{\lambda+\mu}{2}\right) \|\dtx(t)\|^2 - L \left(\gamma_2-\beta(\lambda,\gamma_1,\gamma_2) M_{\lambda,\mu} -\frac{\lambda+\mu}{2} \right) \|\dty(t)\|^2 \leq 0,
\end{align}
where we used that the stepsizes $\gamma_1$ and $\gamma_2$ fulfill  (\ref{g1}). We will use the following notation for the positive constants in the above inequality
\begin{align*}
	m_1:=L \left(\gamma_1 - \beta(\lambda,\gamma_1,\gamma_2) M_{\lambda,\mu} - \frac{\lambda+\mu}{2}\right)  \ \mbox{and} \
	m_2:=L\left(\gamma_2-\beta(\lambda,\gamma_1,\gamma_2) M_{\lambda,\mu} -\frac{\lambda+\mu}{2} \right).
\end{align*}
The decreasing property \eqref{hdec} serves as inspiration for the definition of the Lyapunov functional,  which will be used in the next subsection to show convergence of the trajectory. 
Integrating (\ref{hdec}) from 0 to $T$, with $T>0$ fixed, yields
\begin{align}\label{intdec}
&\Psi(\dtx(T) + x(T), \dty(T)+y(T)) + \frac{1}{2} \left\|\left(\sqrt{\gamma_1 L} \dtx(T), \,\sqrt{\gamma_2 L} \dty(T)\right)\right\|^2 + \int_0^T \left[m_1\|\dtx(t)\|^2 + m_2\|\dty(t)\|^2 \right] \; \md t \notag\\
\leq \ & \Psi(\dtx(0) + x_0, \dty(0)+y_0) + \frac{1}{2} \left\|\left(\sqrt{\gamma_1 L} \dtx(0), \,\sqrt{\gamma_2 L} \dty(0)\right)\right\|^2. 
\end{align}
By letting $T$ converge to $+\infty$, we can easily deduce that
$\dtx \in L^2([0,+\infty); \R^{n})$, $\dty \in L^2([0,+\infty); \R^{m})$, and  therefore 
\begin{align*}
\dtz = (\dtx,\dty)\in  L^2([0,+\infty);\R^n \times \R^m]),
\end{align*}
as well as 
\begin{align}\label{dec}
\frac{\md}{\md t} &\left( \Psi(\dtx(t) + x(t), \dty(t)+y(t)) + \frac{1}{2}\left\|\left(\sqrt{\gamma_1 L} \dtx(t), \,\sqrt{\gamma_2 L} \dty(t)\right)\right\|^2 \right) \leq 0 \ \mbox{for a.e.} \ t \geq 0.
\end{align}
Due to Assumption \ref{A4}, we also obtain $\ddot{z} \in L^2([0,+\infty);\R^n \times \R^m])$. Since for a.e. $t \in [0,\infty)$
\begin{align*}
\frac{\md}{\md t} \|\dtz(t)\|^2 &= 2 (\langle \dtx(t),\, \ddot{x}(t) \rangle + \langle \dty(t),\, \ddot{y}(t) \rangle) \leq \|\dtx(t)\|^2 + \|\ddot{x}(t)\|^2 + \|\dty(t)\|^2 + \|\ddot{y}\|^2,
\end{align*}
it follows from Lemma \ref{fejer2} that
\begin{align*}
\lim_{t \to +\infty} \dtz(t) = 0 \  \text{and therefore also} \
\lim_{t \to +\infty} \dtx(t) = 0 \text{ and } \lim_{t \to +\infty} \dty(t) = 0. 
\end{align*}
From (\ref{dec}) and Lemma \ref{fejer1} we can also conclude that
\begin{align*}
\exists \lim_{t \to +\infty} \Psi(\dtx(t)+x(t), \dty(t) + y(t)) \in \R.
\end{align*}
\end{proof}

\subsection{A Lyapunov functional}

The major aim of this subsection is to prove that the limit set of the solution trajectory is a subset of the set of critical points of the objective function $\Psi$. We recall that the \textit{limit set} of a trajectory $z: [0,+\infty) \to \R^m \times \R^n$ of the dynamical system \eqref{dynpalm} is defined as
$$\omega(z) := \{\bar{z} \in \R^m \times \R^n \; : \; \exists t_k \to +\infty \text{ such that } z(t_k) \to \bar{z} \text{ as } k \to \infty \}.$$ 
The main tool in our analysis is the following \textit{Lyapunov functional} $\mathcal{H}: \R^m \times \R^n \times \R^m \times \R^n \mapsto \bar{\R}$,
\begin{align}\label{ly}
\h\big[(x,y),(u,v)\big] := \Psi(x,y)+\frac{1}{2} \left\|\left(\sqrt{\gamma_1 L}(x-u), \sqrt{\gamma_2 L}(y-v)\right)\right\|^2, 
\end{align}
the definition of which is obviously inspired by (\ref{dec}). Note that
\begin{align*}
\pa\h\big[(x,y),(u,v)\big] = \Big(\pa\Psi(x,y) + L\left(\gamma_1 (x-u), \gamma_2 (y-v)\right)\Big) \times \Big\{-L\Big(\gamma_1 (x-u), \gamma_2 (y-v)\Big)\Big\}
\end{align*}
for every $(x,y), (u,v) \in \R^m \times \R^n$, thus
\begin{align*}
	& \pa\h \big[(\dtx(t)+x(t),\dty(t)+y(t)),(x(t),y(t))\big]  \notag \\
	 = \  &\Big(\pa\Psi((\dtx(t)+x(t),\dty(t)+y(t)) + L\left(\gamma_1 \dtx(t), \gamma_2 \dty(t)\right)\Big) \times \Big\{-L \Big(\gamma_1\dtx(t), \gamma_2 \dty(t)\Big)\Big\} \ \mbox{for a.e.} \ t \geq 0.
\end{align*}
Instead of investigating the trajectory $z$ itself, we go over to $(\dtz+z,z)$. We can prove the following properties of the Lyapunov functional $\h$.
\begin{theorem}\label{hthm}
	Let $f$, $g$ and $H$ fulfill Assumption \ref{A1} and $f$, $g$ admit a chain rule. Let further $z=(x,y) : [0,+\infty) \rightarrow \R^n \times \R^m$ be a strong global solution of the dynamical system (\ref{dynpalm}) fulfilling Assumption \ref{A4}, where the constants are chosen such that \eqref{g1} holds. 
	Then the following statements are true:
	\begin{description}
		\item[(i) decrease of $\h$:] for a.e. $t \geq 0$ it holds
		\begin{align}\label{dech}
			&\frac{\md}{\md t} \h\bigg[\big(\dtx(t)+x(t), \dty(t)+y(t)\big), \big(x(t), y(t)\big)\bigg] \leq -\min{\{m_1,m_2\}} \|(\dtx(t),\dty(t))\|^2 \leq 0,
		\end{align}
		as well as
		\begin{align}\label{limh}
			\exists \lim_{t \to \infty} \h\big[(\dtx(t)+x(t), \dty(t)+y(t)), (x(t),y(t))\big] \in \R.
		\end{align}
		\item[(ii) subgradient lower bound:] for a.e. $t \geq 0$ it holds
		\begin{align}\label{subgrad}
		&\left( \begin{pmatrix} \nabla_x H(\dtx(t)+ x(t), \dty(t)+ y(t))- \nabla_x H(x(t), (1-\mu)(\dty(t)+y(t))+\mu y(t)) \\  \nabla_y H(\dtx(t)+ x(t), \dty(t)+ y(t))-\nabla_y H((1-\lambda)(\dtx(t)+x(t))+\lambda x(t), y(t))\end{pmatrix}, \, 
		- L \begin{pmatrix} \gamma_1 \dtx(t)\\ \gamma_2 \dty(t) \end{pmatrix} 
		\right) \notag \\
			\in \ & \pa\h\big[(\dtx(t)+x(t), \dty(t)+y(t)), (x(t)+y(t))\big]
		\end{align}
		and 
		\begin{align}\label{subgradest}
			&\bigg\|\bigg( \begin{pmatrix} \nabla_x H(\dtx(t)+ x(t), \dty(t)+ y(t))- \nabla_x H(x(t), (1-\mu)(\dty(t)+y(t))+\mu y(t)) \\  \nabla_y H(\dtx(t)+ x(t), \dty(t)+ y(t))-\nabla_y H((1-\lambda)(\dtx(t)+x(t))+\lambda x(t), y(t))\end{pmatrix}, \, 
			-L \begin{pmatrix} \gamma_1 \dtx(t)\\ \gamma_2 \dty(t) \end{pmatrix} 
			\bigg)\bigg\| \notag \\
			& \leq L\sqrt{\max{\left\{(1+\gamma_1^2 +\lambda^2), (1 + \gamma_2^2 +\mu^2)\right\}}} \|(\dtx(t), \dty(t))\|.
		\end{align}
		\item[(iii) convergence to a critical point:] it
		\begin{align}\label{subset}
		\omega(z) \subset \crit( \Psi) 
		\end{align}
		and for every $\bar z = (\bar{x},\bar{y}) \in \omega(z)$  and $t_k \to +\infty$ such that $z(t_k) = (x(t_k), y(t_k)) \to \bar{z}$  as $k \to +\infty$, we have 
		\begin{align}\label{convcrit}
			\h\big[(\dtx(t_k)+x(t_k),\dty(t_k)+y(t_k)),(x(t_k),y(t_k))\big] \to \h\big[(\bar{x}, \bar{y}),(\bar{x}, \bar{y})\big] \text{ as } k \to +\infty.
		\end{align} 
	\end{description}
\end{theorem}
 \begin{proof}
 	\textbf{(i) decrease of $\h$:} Relation (\ref{dech}) follows directly from (\ref{hdec}), while (\ref{limh}) can be deduced from Theorem \ref{theorem1} (2).

 	\textbf{(ii) subgradient lower bound:} Adding $\nabla_xH(\dtx(t)+x(t), \dty(t)+y(t))$ to (\ref{paf}) and \newline $\nabla_yH(\dtx(t)+x(t), \dty(t)+y(t))$ to (\ref{pag}) one obtains for a.e. $t \geq 0$
 	\begin{align}\label{useful}
	 	& -\gamma_1 L \dtx(t) +\nabla_xH(\dtx(t)+x(t), \dty(t)+y(t)) -\nabla_xH(x(t),(1-\mu)(\dty(t)+y(t))+\mu y(t)) \notag\\
	 	\in \ & \pa_x (\tilde{f} + H) (\dtx(t)+x(t), \dty(t)+y(t))
	\end{align}
and, respectively,
	\begin{align}\label{useful2}
	 	-\gamma_2 L \dty(t) &+\nabla_yH(\dtx(t)+x(t), \dty(t)+y(t)) -\nabla_yH((1-\lambda)(\dtx(t)+x(t))+\lambda x(t),y(t)) \notag\\
	 	\in \ & \pa_y (\tilde{g}+H)(\dtx(t)+x(t), \dty(t)+y(t)),
 	\end{align}
where $\tilde f : \R^n \times \R^m \rightarrow \overline \R, \tilde{f}(x,y):=f(x)$ and $\tilde g : \R^n \times \R^m \rightarrow \overline \R, \tilde{g}(x,y):=g(y)$. Since we have 
 	$$\pa \Psi(z) = \bigg(\pa_x(\tilde{f}+H)(z), \pa_y(\tilde{g}+H)(z)\bigg) \ \mbox{for all} \ z \in \R^m \times \R^n,$$
 	we can conclude
 	\begin{align*}
	 	&\begin{pmatrix} \nabla_xH(\dtx(t)+x(t), \dty(t)+y(t)) -\nabla_xH(x(t),(1-\mu)(\dty(t)+y(t))+\mu y(t)) \\ \nabla_yH(\dtx(t)+x(t), \dty(t)+y(t)) -\nabla_yH((1-\lambda)(\dtx(t)+x(t))+\lambda x(t),y(t)) \end{pmatrix}  \\
	 	\in \ & \pa \Psi \begin{pmatrix} \dtx(t)+x(t) \\ \dty(t)+y(t) \end{pmatrix} + L \begin{pmatrix} \gamma_1 \dtx(t) \\  \gamma_2 \dty(t) \end{pmatrix} \ \mbox{for a.e.} \ t \geq 0,
 	\end{align*}
 which is equivalent to (\ref{subgrad}). Furthermore, we can estimate the subgradient using the Lipschitz-continuity of $\nabla H$ and obtain for a.e. $t \geq 0$
 	\begin{align*}
	 	&\left\| \left(\begin{pmatrix} \nabla_x H(\dtx(t)+ x(t), \dty(t)+ y(t)) - \nabla_x H(x(t), (1-\mu)(\dty(t)+y(t))+\mu y(t)) \\  \nabla_y H(\dtx(t)+ x(t), \dty(t)+ y(t)) - \nabla_y H((1-\lambda)(\dtx(t)+x(t))+\lambda x(t), y(t))\end{pmatrix}, \, 
	 	-\begin{pmatrix} \gamma_1 L \dtx(t) \\ \gamma_2 L\dty(t)\end{pmatrix}
	 	\right) \right\|^2 \\
	 	= \ &  \|\nabla_xH(\dtx(t)+x(t), \dty(t)+y(t)) - \nabla_x H(x(t), (1-\mu)(\dty(t)+y(t))+\mu y(t))\|^2 \\
	 	&+ \|\nabla_yH(\dtx(t)+x(t), \dty(t)+y(t)) - \nabla_y H((1-\lambda)(\dtx(t)+x(t))+\lambda x(t), y(t))\|^2 \\
	 	&+ \gamma_1^2 L^2 \| \dtx(t) \|^2 + \gamma_2^2 L^2 \|\dty(t))\|^2 \\
	 	\leq \ & L^2 (1+\gamma_1^2 +\lambda^2) \|\dtx(t)\|^2 + L^2 (1 + \gamma_2^2 +\mu^2) \|\dty(t)\|^2,
 	\end{align*}
 	from which statement (\ref{subgradest}) follows. 

 	\textbf{(iii) convergence to a critical point:}  Let $(\bar{x},\bar{y}) \in \omega(x,y)$ and $(t_k)_{k \geq 0}$ be such that  $(x(t_k),y(t_k))$ converges to $(\bar{x},\bar{y})$ as $k \to +\infty$. Due to Theorem \ref{theorem1} it also holds $\lim_{k \to +\infty}(\dtx(t_k),\dty(t_k)) = (0,0)$. We claim that 
 	$$\lim_{k \to +\infty} (\tilde{f}+H)(\dtx(t_k)+x(t_k), \dty(t_k)+y(t_k))  = (\tilde{f}+H)(\bar{x}, \bar{y})$$
 	and 
 	$$\lim_{k \to +\infty} (\tilde{g}+H)(\dtx(t_k)+x(t_k), \dty(t_k)+y(t_k))  = (\tilde{g}+H)(\bar{x}, \bar{y}).$$
 	In fact, since $f$ and $g$ are assumed to be lower semicontinuous, we have 
 	$$\liminf_{k \to +\infty} f(\dtx(t_k)+x(t_k)) \geq f(\bar{x})$$ 
 	and 
 	$$\liminf_{k \to +\infty} g(\dty(t_k)+y(t_k)) \geq g(\bar{y}).$$
 	On the other hand, we can conclude from the definition of the dynamical system (\ref{dynpalm}) that for any $k \geq 0$
 	\begin{align*}
 		&\dtx(t_k)+x(t_k) \\
 		\in &\argmin\limits_{u \in \R^n} \left\{f(u) + \frac{\gamma_1 L}{2} \|u-x(t_k)\|^2 + \Big\langle u - x(t_k), \nabla_xH(x(t_k), (1-\mu)(\dty(t_k)+y(t_k))+\mu y(t_k)) \Big\rangle \right\}
 	\end{align*}
 	and 
 	\begin{align*}
 	& \dty(t_k)+y(t_k) \\
 	\in & \argmin_{v \in \R^m} \left\{g(v) + \frac{\gamma_2 L}{2} \|v-y(t_k)\|^2 + \Big\langle v - y(t_k), \nabla_yH((1-\lambda)(\dtx(t_k)+ x(t_k))+\lambda x(t_k), y(t_k)) \Big\rangle \right\},
 	\end{align*}
 	which implies 
 	\begin{align*}
 	&f(\dtx(t_k)+x(t_k)) + \frac{\gamma_1 L}{2} \|\dtx(t_k)\|^2 + \Big\langle \dtx(t_k), \nabla_xH(x(t_k), (1-\mu)(\dty(t_k)+y(t_k))+\mu y(t_k)) \Big\rangle \\
 	\leq \ & f(\bar{x}) + \frac{\gamma_1 L}{2} \|\bar{x}-x(t_k)\|^2 + \Big\langle \bar{x}-x(t_k), \nabla_xH( x(t_k), (1-\mu)(\dty(t_k)+y(t_k))+\mu y(t_k)) \Big\rangle ,
 	\end{align*}
 	and
 	\begin{align*}
 		&g(\dty(t_k)+y(t_k)) + \frac{\gamma_2 L}{2} \|\dty(t_k)\|^2 + \Big\langle \dty(t_k), \nabla_yH(\dtx(t_k)+ x(t_k), y(t_k)) \Big\rangle \\
 		\leq \ & g(\bar{y}) + \frac{\gamma_2 L}{2} \|\bar{y}-y(t_k)\|^2 + \Big\langle \bar{y}-y(t_k), \nabla_yH((1-\lambda)(\dtx(t_k)+ x(t_k))+\lambda x(t_k), y(t_k)) \Big\rangle .
 	\end{align*}
 	Therefore we can conclude under consideration of the continuity of $\nabla H$
 	$$\limsup_{k \to \infty} f(\dtx(t_k)+x(t_k)) \leq f(\bar{x})$$
 	as well as
 	$$\limsup_{k \to \infty} g(\dty(t_k)+y(t_k)) \leq g(\bar{y}),$$
 	from which the claim follows. Using again continuity of $\nabla H$ we obtain for $k \to +\infty$
 	\begin{align*}
	 -\gamma_1 L \dtx(t_k) +\nabla_xH(\dtx(t_k)+x(t_k), \dty(t_k)+y(t_k)) -\nabla_xH(x(t_k),(1-\mu)(\dty(t_k)+y(t_k))+\mu y(t_k)) \rightarrow 0, \\
 		-\gamma_2 L \dty(t_k) +\nabla_yH(\dtx(t_k)+x(t_k), \dty(t_k)+y(t_k)) -\nabla_yH((1-\lambda)(\dtx(t_k)+x(t_k))+\lambda x(t_k),y(t_k))\rightarrow 0.
 	\end{align*}
Using \eqref{useful} and \eqref{useful2}, and the closedness of the graph of the limiting subdifferential, it yields
 	$$0 \in \pa_x (\tilde{f}+H)(\bar{x},\bar{y}) \text{ and } 0 \in \pa_y (\tilde{g}+H)(\bar{x}, \bar{y}),$$
 	which is equivalent to 
 	$$0 \in \pa \Psi (\bar{x},\bar{y})$$
 	and (\ref{subset}) is proven. The statement in (\ref{convcrit}) follows immediately.
 \end{proof}
 \begin{remark}
 The limit set $\omega(z)$ of the trajectory $z$ is not empty if $z(t), t \in [0,+\infty)),$ is bounded. This follows, for example, if the objective function $\Psi$ is \emph{coercive}, i.e.
 	$$\lim_{\|z\| \to +\infty} \Psi(z) = +\infty.$$
 Furthermore, in this case it also follows that $\Psi$ is bounded from below, the infimum is attained and all its sublevel sets are bounded. Furthermore, from (\ref{intdec}), again with the notation $z(t)=(x(t), y(t))$, we have for all $T>0$
 	\begin{align*}
 	\Psi(\dtx(T) + x(T), \dty(T)+y(T)) \leq \Psi(\dtx(0) + x_0, \dty(0)+y_0) + \frac{1}{2} \left\|\left(\sqrt{\gamma_1 L} \dtx(0), \,\sqrt{\gamma_2 L} \dty(0)\right)\right\|^2. 
 	\end{align*}
 Due to the boundedness of the sublevel sets of $\Psi$, we can conclude from the above estimate that $\dtz +z$ is bounded.  Since $\lim_{t \to \infty} \dtz(t) = 0$, it follows that $z$ is bounded.   
 \end{remark}
 
\begin{cor}\label{hcor}
Let $f$, $g$ and $H$ fulfill Assumption \ref{A1} and $f$, $g$ admit a chain rule. Let further $z=(x,y) : [0,+\infty) \rightarrow \R^n \times \R^m$ be a strong global solution of the dynamical system (\ref{dynpalm}) assumed to be bounded and fulfilling Assumption \ref{A4}, where the constants are chosen such that \eqref{g1} holds. Then the following statements are true:
	\begin{enumerate}
		\item $\omega(\dtz+z, z) \subseteq \crit(\h) = \{(u,u) \in \R^m \times \R^n \times \R^m \times \R^n :\; u \in \crit(\Psi) \}$;
		\item  $\lim_{t \to +\infty} \dist\big((\dtz(t)+z(t),z(t)),\omega(\dtz+z,z)\big)=0$;
		\item  the set $\omega(\dtz+z,z)$ is nonempty, compact and connected;
		\item $\h$ is finite and constant on $\omega(\dtz+z,z)$.
	\end{enumerate}
\end{cor}
\begin{proof}
	The statements 1., 2. and 4. are direct consequences of Theorem \ref{hthm}. For the proof of statement 3. we refer the reader to \cite[Theorem 4.1.]{al}, where it is shown that the non-emptiness, compactness and connectedness of the limit set of a trajectory is a generic property for bounded trajectories, whose derivatives converge to $0$ for $t \to +\infty$.
\end{proof}

\begin{remark}
	In the special case when $\mu=\lambda=0$ one can prove that the statements in Theorem \ref{theorem1}, Theorem \ref{hthm} and also in Corollary \ref{hcor} remain true under the weaker assumption on $\nabla H$ that its partial gradients $\nabla_x H( \cdot ,y)$ and $\nabla_yH(x, \cdot)$ are Lipschitz continuous with global  constants $L_x$ and $L_y$ in $y$, respectively, in $x$.  This allows us to choose in the in the two inclusions of the dynamical system \eqref{dynpalm} as stepsizes $\gamma_1L_x$ and $\gamma_2L_y$, respectively.  This means that the dynamical system is of the form
	\begin{equation}
	\begin{cases}
	\dtx(t)+x(t) \in \prox_{\frac{1}{\gamma_1 L_x}f}\left(x(t)-\frac{1}{\gamma_1 L_x}\nabla_xH(x(t), \dty(t)+y(t))\right), \\
	\dty(t)+y(t)\in\prox_{\frac{1}{\gamma_2 L_y}g}\left(y(t)-\frac{1}{\gamma_2 L_y}\nabla_yH(\dtx(t)+x(t), y(t))\right), \\
	(x(0),y(0))=(x_0,y_0).
	\end{cases}
	\end{equation}
and has as discrete counterpart the PALM algorithm from \cite{bol}.

The essential ingredient in the asymptotic analysis is proof of a decreasing property for the function $t \mapsto \frac{\md}{\md t} \Psi(\dtx(t) + x(t), \dty(t)+y(t))$. While in the proof of Theorem \ref{theorem1} we used to this end the Lipschitz continuity of the full gradient $\nabla H$ when estimating the time derivative of $\Psi$ along the trajectory $(x(t)+\dtx(t), y(t)+\dty(t))$, we now obtain
	\begin{align*}
	\frac{\md}{\md t} \Psi(\dtx(t) + x(t), \dty(t)+y(t))
	 \leq  & -\gamma_1 L_x \|\dtx(t)\|^2 -\gamma_2 L_y \|\dty(t)\|^2 - \frac{\gamma_1 L_x}{2} \frac{\md}{\md t} \|\dtx(t)\|^2 - \frac{\gamma_2 L_y}{2} \frac{\md}{\md t} \|\dty(t)\|^2 \notag \\
& +L_x \|\dtx(t)\| \|\ddot{x}(t) + \dtx(t)\| +L_y \|\dty(t)\| \|\ddot{y}(t) + \dty(t)\| \ \mbox{for a.e.} \ t \geq 0.
	\end{align*}
	Therefore, (\ref{est3}) simplifies to
	\begin{align*}
	L_x \|\dtx(t)\| \|\ddot{x}(t) + \dtx(t)\| +L_y \|\dty(t)\| \|\ddot{y}(t) + \dty(t)\|
	&\leq  \sqrt{\|\ddot{x}(t)\|^2+\|\ddot{y}(t)\|^2} \sqrt{L_x^2 \|\dtx(t)\|^2+ L_y^2\|\dty(t)\|^2} \\
	&= \max{\{L_x,L_y\}} \|(\ddot{x}(t),\ddot{y}(t))\| \|(\dtx(t),\dty(t)\|\\
& \leq  \beta(\lambda, \gamma_1, \gamma_2) \max{\{L_x,L_y\}} \|(\dtx(t), \dty(t))\|^2 \! \ \mbox{for a.e.} \! \ t \geq 0,
	\end{align*}
	where $ \beta(\lambda, \gamma_1, \gamma_2)$ describes the constant from Assumption \ref{A4}. Proceeding in the same way as in the proof of Theorem \ref{theorem1}, we obtain, instead of (\ref{hdec}), the following decreasing property for a.e. $t \geq 0$
	\begin{align}
	& \frac{\md}{\md t} \left( \Psi(\dtx(t) + x(t), \dty(t)+y(t)) + \frac{1}{2}\left\|\left(\sqrt{\gamma_1 L_x} \dtx(t), \,\sqrt{\gamma_2 L_y} \dty(t)\right)\right\|^2 \right) \notag \\
	\leq & \ - \left(L_x\gamma_1 - \max{\{L_x,L_y\}}\beta  \right) \|\dtx(t)\|^2 -  \left(L_y \gamma_2-\max{\{L_x,L_y\}}\beta \right) \|\dty(t)\|^2 \\
	\leq & \ 0, \notag
	\end{align}
for $\gamma_1$ and $\gamma_2$ chosen such that
	$$\gamma_1 > \frac{\max{\{L_x,L_y\}} \beta(\lambda, \gamma_1, \gamma_2)}{L_x} \ \mbox{and} \ \gamma_2 > \frac{\max{\{L_x,L_y\}} \beta(\lambda, \gamma_1, \gamma_2)}{L_y}.$$
	
\end{remark}

\subsection{Global convergence through the KL property}

In this subsection we will show that if the Lyapunov function satisfies the \textit{Kurdyka-\L ojasiewicz (KL) property}, then  one can establishes \textit{global convergence} of the trajectory to a critical point of $\Psi$. 

To recall the definition of the KL property,  let us denote by $\Theta_\eta$, for $\eta \in [0,+\infty]$, the set of all concave and continuous functions $\vp: [0,\eta) \to \R$, which satisfy the following properties:
\begin{itemize}
	\item $\vp(0)=0$;
	\item $\vp \in C^1((0,\eta))$ and $\vp$ is continuous at 0;
	\item $\forall s \in (0,\eta): \vp'(s)>0$.
\end{itemize}

\begin{Def}[Kurdyka-\L ojasiewicz property]\label{KL}
	Let $h: \R^d \to \overline{\R}$ be proper and lower semicontinuous.
	\begin{enumerate}
		\item The function $h$ is said to satisfy the Kurdyka-\L ojasiewicz (KL) property at a point $\bar{u} \in \dom \pa h$, if there exist $\eta \in (0,+\infty]$, a neighborhood $U$ of $\bar{u}$ and a function $\vp \in \Theta_\eta$ such that for every
		$$u \in U \cap \{u \in \R^d: \: h(\bar{u})<h(u)<h(\bar{u})+\eta\}$$
		the inequality 
		$$\vp'\Big(h(u)-h(\bar{u})\Big) \dist\Big(0,\pa h (u)\Big)\geq 1$$
		holds. The function $\vp$ is called \emph{desingularizing function} for $h$ at point $\bar{u}$.
		\item If $h$ satisfies the KL property at every point in $\dom \pa h$, then $h$ is said to be a\emph{KL function}.
	\end{enumerate}
\end{Def}
To the class of KL functions belong semi-algebraic, real sub-analytic, semiconvex, uniformly
convex and convex functions satisfying a growth condition, see \cite{att3, bol2, bol}. However, for our analysis we will make use of the \textit{uniform KL property}, which follows directly from the above definition, see \cite[Lemma 3.6]{bol}.
\begin{lemma}\label{kl}
	Let $\Omega \subset \R^d$ be a compact set and $h: \R^d \to \overline \R$ a proper and lower semicontinuous function. Assume that $h$ is constant on $\Omega$ and it satisfies the KL property at each point of $\Omega$.  Then there exist $\ve, \eta >0$ and $\vp \in \Theta_\eta$ such that  for every $\bar{u} \in \Omega$ and every $u$ in the intersection
	$$\{u \in \R^d: \; \dist(u,\Omega)<\ve\} \cap \{u \in \R^d: \; h(\bar{u})<h(u)<h(\bar{u})+\eta\}$$
	the inequality
	$$\vp'\Big(h(u)-h(\bar{u})\Big)\dist\Big(0,\pa h(u)\Big)\geq 1$$
	holds.
\end{lemma}
With these preparations we are now ready to state the main result:
\begin{theorem}\label{mainresult}
			Let $f$, $g$ and $H$ fulfill Assumption \ref{A1} and $f$, $g$ admit a chain rule. Let further $z=(x,y) : [0,+\infty) \rightarrow \R^n \times \R^m$ be a strong global solution of the dynamical system (\ref{dynpalm}) assumed to be bounded and fulfilling Assumption \ref{A4}, where the constants are chosen such that \eqref{g1} holds. Moreover, suppose that the Lyapunov function $\h$, defined as in (\ref{ly}), is a KL function. Then the following statements hold:
	\begin{enumerate}
		\item $\dtz=(\dtx, \dty) \in L^1([0,\infty); \R^n \times \R^m)$;
		\item there exists $\bar{z}=(\bar{x},\bar{y}) \in \crit(\Psi)$ such that $\lim_{t \to +\infty} z(t)=\lim_{t \to +\infty} (x(t),y(t)) = (\bar{x},\bar{y})=\bar{z}.$
	\end{enumerate}
\end{theorem}
\begin{proof}
	Choose $\bar{z} \in \crit{(\Psi)}$ such that $(\bar{z},\bar{z}) \in \omega(\dtz+z,z)$. From Theorem \ref{hthm} and Corollary \ref{hcor} we also have $\lim_{t \to +\infty} \h[\dtz(t)+z(t),z(t)] = \h[\bar{z},\bar{z}]$ and $\bar{z} \in \crit{(\Psi)}$. We distinguish between two cases:
	\begin{description}
		\item[Case I:] {\bf $\exists \bar{t} \geq 0$ such that $\h[\dtz(\bar{t})+z(\bar{t}),z(\bar{t})] = \h[\bar{z},\bar{z}]$.}  

Since from Theorem \ref{hthm} (i) we have that $\h$ is non-increasing along trajectories, we obtain for all $t \geq \bar{t}$
		$$\h[\dtz(t)+z(t),z(t)] \leq \h[\dtz(\bar{t})+z(\bar{t}),z(\bar{t})] = \h[\bar{z},\bar{z}],$$
		and therefore 
		$$t \mapsto \h[\dtz(t)+z(t),z(t)]$$ 
		is constant on $[\bar{t}, +\infty)$. Due to the subgradient estimate in Theorem \ref{hthm} (ii), we obtain for a.e. $t \in [\bar{t},+\infty)$ that $\dtz(t)=(\dtx(t),\dty(t))=0$ and, hence, $z(t)=(x(t),y(t))$ is constant on $[\bar{t},+\infty)$, from which the conclusion follows.
		\item[Case II:] {\bf $\forall t \geq 0: \; \h[\dtz(t)+z(t), z(t)]>\h[\bar{z},\bar{z}]$.}

		We use Lemma \ref{kl} for $\Omega :=\omega(\dtz+z,z)$ and $h := \h$. According to it, there exist $\ve, \eta >0$ and $\vp \in \Theta_\eta$ such that for all $(\bar{z},\bar{z}) \in \omega(\dtz+z,z)$ and all $(v,w)$ in the intersection
		\begin{align*}
			\mathcal{A}:= & \{(v,w) \in\R^n \times \R^m \times \R^n \times \R^m : \; \dist((v,w),\omega(\dtz+z,z))<\ve\} \ \cap \\
			& \{(v,w) \in\R^n \times \R^m \times \R^n \times \R^m : \; \h[\bar{z},\bar{z}]<\h[v,w]<\h[\bar{z},\bar{z}]+\eta\}
		\end{align*}
		one has 
		\begin{align*}
			\vp'\Big(\h[v,w]-\h[\bar{z},\bar{z}]\Big)\dist{\Big((0,0),\pa \h[v,w]\Big)} \geq 1.
		\end{align*}
	\end{description}
	On one hand, since 
	$$\lim_{t \to +\infty} \h[\dtz(t)+z(t),z(t)]= \h[\bar{z},\bar{z}],$$
there exists	$t_1 \geq 0$ such that 
	$$\h[\bar{z},\bar{z}]<\h[v,w]<\h[\bar{z},\bar{z}]+\eta$$ 
	for every $t \geq t_1$. On the other hand, from 
	$$\lim_{t \to +\infty} \dist{\Big((\dtz(t)+z(t),z(t)),\omega(\dtz +z, z)\Big)}=0,$$
	there exist $t_2\geq 0$ such that 
	$$\dist{\Big((\dtz(t)+z(t),z(t)),\omega(\dtz +z, z)\Big)}<\ve.$$
	Therefore, for a.e. $t \geq T:=\max{\{t_1,t_2\}}$ it holds 
	$$\Big(\dtz(t)+z(t),z(t)\Big) \in \mathcal{A},$$
	and hence
	$$\vp'\Big(\h[\dtz(t)+z(t),z(t)]-\h[\bar{z},\bar{z}]\Big)\dist{\Big((0,0),\pa \h[\dtz(t)+z(t),z(t)]\Big)} \geq 1.$$
	From here we can conclude using  (\ref{subgrad}) and (\ref{subgradest}) that
	\begin{align}\label{estimate1}
		1 \leq \vp'\Big(\h[\dtz(t)+z(t),z(t)]-\h[\bar{z},\bar{z}]\Big)  L\sqrt{\max{\left\{(1+\gamma_1^2 +\lambda^2), (1 + \gamma_2^2 +\mu^2)\right\}}} \|(\dtz(t)\| \ \mbox{for a.e.} \ t \geq T.
	\end{align}
This yields the following estimation for the time derivative
	\begin{align}\label{desest}
	&\frac{\md}{\md t}\vp\Big(\h[\dtz(t)+z(t),z(t)]-\h[\bar{z},\bar{z}]\Big) \notag \\
	= & \ \vp'\Big(\h[\dtz(t)+z(t),z(t)]-\h[\bar{z},\bar{z}]\Big) \frac{\md}{\md t} \h[\dtz(t)+z(t),z(t)] \\
	\leq & \-\min{\{m_1, m_2\}}\left(L\sqrt{\max{\left\{(1+\gamma_1^2 +\lambda^2), (1 + \gamma_2^2 +\mu^2)\right\}}}\right)^{-1}\|\dtz(t)\| \leq 0 \notag \ \mbox{for a.e.} \ t \geq T,
	\end{align}
which is due to the decreasing property (\ref{dech}) of the Lyapunov-function $\h$ and inequality (\ref{estimate1}). After integration, since $\vp$ is bounded from below, we can deduce
	$$\dtz \in L^1([0,\infty);\R^n \times \R^m).$$
This yields that the limit $\lim_{t \to +\infty} z(t)$ exists, from which the conclusion follows. 
\end{proof}

\subsection{Convergence rates}

This subsection is dedicated to the derivation convergence rates of the trajectory of the dynamical system (\ref{dynpalm}) to a critical point of $\Psi$. Convergece rate sesults (see \cite{ab}) can be achieved considering a KL function, which satisfy Definition \ref{KL} with desingularizing function 
\begin{align}\label{des}
	\vp(s)=\frac{1}{c} s^{1-\theta}, \quad \theta \in [0,1),\;  c >0.
\end{align}

\begin{Def}[\L ojasiewicz property]
	Let $h: \R^d \to \overline{\R}$ be proper and lower semicontinuous.
\begin{enumerate}
		\item The function $h$ is said to satisfy the \L ojasiewicz property at a point $\bar{u} \in \dom \pa h$, if there exist $\eta \in (0,+\infty]$, a neighborhood $U$ of $\bar{u}$, $c >0$ and $\theta \in [0,1)$ such that for every
$$u \in U \cap \{u \in \R^d: \: h(\bar{u})<h(u)<h(\bar{u})+\eta\}$$ 
it holds
$$|h(u) - h(\bar u)|^\theta \leq c \|\xi\| \ \mbox{for every} \ \xi \in \pa h(u).$$
The exponent $\theta$ is called the \textit{\L ojasiewicz exponent}.
		\item If $h$ has the KL property and has the same \L ojasiewicz exponent $\theta$ at every point in $\dom \pa h$, then we say that $h$ is a \L ojasiewicz function with an exponent of $\theta$.
	\end{enumerate}
\end{Def}
 It is known that analytic functions and proper lower semicontinuous semi-algebraic functions have the  \L ojasiewicz property  (see \cite{att3, loj}). Calculus rules for the \L ojasiewicz exponent have been developed in \cite{lipong}.

\begin{theorem}
		Let $f$, $g$ and $H$ fulfill Assumption \ref{A1} and $f$, $g$ admit a chain rule. Let further $z=(x,y) : [0,+\infty) \rightarrow \R^n \times \R^m$ be a strong global solution of the dynamical system (\ref{dynpalm}) assumed to be bounded and fulfilling Assumption \ref{A4}, where the constants are chosen such that \eqref{g1} holds. Moreover, suppose that the Lyapunov function $\h$, defined as in (\ref{ly}), has the \L ojasiewicz property. Then there exists $\bar{z} = (\bar{x}, \bar{y}) \in \crit{(\Psi)}$ such that $\lim_{t \to +\infty} z(t)=\lim_{t \to +\infty} (x(t),y(t)) = (\bar{x},\bar{y})=\bar{z}$. Let $\theta \in [0,1)$ be the \L ojasiewicz exponent of $\h$ at $(\bar{z},\bar{z})$. Then there exist $\alpha, \beta, \gamma, \delta >0$ and a time $T \geq 0$ such the following statements are true:
	\begin{enumerate}
		\item if $\theta \in \left[0, \frac{1}{2}\right)$, then $z=(x,y)$ converges in finite time;
		\item if $\theta = \frac{1}{2}$, then $\|z(t)-\bar{z}\| \leq \alpha e^{-\beta t} \ \forall t \geq T$; 
		\item if $\theta \in \left(\frac{1}{2}, 1\right)$, then $\|z(t)-\bar{z}\| \leq (\gamma t+\delta)^{-\frac{1-\theta}{2\theta -1}} \ \forall t \geq T$.
	\end{enumerate} 
\end{theorem}

\begin{proof}
	If there exists $\bar{t} \geq 0$ such that $\h[\dtz(\bar{t})+z(\bar{t}),z(\bar{t})] = \h[\bar{z},\bar{z}]$, then the trajectory converges in finite time, as we saw in the first case of the proof of Theorem  \ref{mainresult}.

Let us therefore assume that for all $t\geq 0$ the inequality $\h[\dtz(t)+z(t), z(t)]>\h[\bar{z},\bar{z}]$ holds. We define the \textit{arc-length} of the trajectory $z$ from time $t \geq 0$ on by
	\begin{align*}
	\sigma(t) := \int_{t}^{+\infty} \|\dtz(s)\| \; \md s.
	\end{align*}
	We immediately obtain
	\begin{align}\label{diffarc}
	\|\dtz(t)\| = -\dot{\sigma}(t) \ \mbox{for a.e.} \ t \geq 0,
	\end{align}
	and by calculating for some time $\tau \geq t$
	\begin{align*}
	\|z(t)-\bar{z}\| = \left\| z(\tau) - \bar{z} - \int_t^\tau \dtz(s) \; \md s \right\| \leq \|z(\tau) - \bar{z}\| + \int_t^\tau \|\dtz(s)\| \; \md s,
	\end{align*}
	and then taking the limit $\tau \to +\infty$, we obtain the following estimate for the distance of the trajectory to its limit point $\bar z$
	\begin{align*}
	\|z(t)-\bar{z}\| \leq \sigma(t) \ \mbox{for all} \ t \geq 0.
	\end{align*}
	From the considerations in (\ref{desest}) in the proof of Theorem \ref{mainresult}, while taking into account the special shape of the desingularizing function (\ref{des}), we get that there exists $\; T >0$ such that for a.e $t \geq  T$
	\begin{align*}
		M_1\|\dtz(t)\| + c \frac{d}{dt}\Big(\h[\dtz(t)+z(t), z(t)]-\h[\bar{z},\bar{z}]\Big)^{1-\theta} \leq 0,
	\end{align*}
	where we define
	$$M_1:=\min{\{m_1, m_2\}} \left(L\sqrt{\max{\left\{(1+\gamma_1^2 +\lambda^2), (1 + \gamma_2^2 +\mu^2)\right\}}}\right)^{-1}.$$
	Integration from $t$ to $\widetilde{T}$, where $\widetilde{T} \geq T$, yields
	\begin{align*}
		M_1 \int_t^{\widetilde{T}} \|\dtz(s)\| \; \md s + c  \Big(\h[\dtz(\widetilde{T})+z(\widetilde{T}), z(\widetilde{T})] - \h[\bar{z},\bar{z}]\Big)^{1-\theta} \leq c \left(\h[\dtz(t)+z(t), z(t)] - \h[\bar{z},\bar{z}]\right)^{1-\theta},
	\end{align*}
	which, after letting $\tilde{T} \to +\infty$, produces the estimate
	\begin{align}\label{arcest2}
		M_1 \sigma(t) \leq c \Big(\h[\dtz(t)+z(t), z(t)] - \h[\bar{z},\bar{z}]\Big)^{1-\theta}
	\end{align}
for a.e $t \geq  T$. In Theorem  \ref{hthm} (ii) we provided an element $(\xi(t), \eta(t)) \in \pa \h[\dtz(t)+z(t), z(t)]$ such that 
	\begin{align*}
		\|(\xi(t), \eta(t))\| \leq M_2\|\dtz(t)\| \quad \ \mbox{for a.e.} \ t \geq 0,
	\end{align*}
	where we defined
	$$M_2:=L\sqrt{\max{\left\{(1+\gamma_1^2 +\lambda^2), (1 + \gamma_2^2 +\mu^2)\right\}}}.$$
Using that $\h$ has the  \L ojasiewicz property $\theta$ at $(\bar{z},\bar{z})$ we further have that for a.e. $t \geq T$ (after potentially increasing $T$ in order to guarantee that $(\dtz(t)+z(t), z(t))$ is in the neighborhood on which the \L ojasiewicz property is fulfilled)
	\begin{align*}
		\Big(\h[\dtz(t)+z(t), z(t)] - \h[\bar{z}, \bar{z}]\Big)^\theta \leq c \|(\xi(t), \eta(t))\| \leq c M_2\|\dtz(t)\|.
	\end{align*}
	Combining this with estimate (\ref{arcest2}) gives
	\begin{align*}
		M_1 \sigma(t) \leq c^{\frac{1}{\theta}} M_2^{\frac{1-\theta}{\theta}} \|\dtz(t)\|^{\frac{1-\theta}{\theta}} \quad \ \mbox{for a.e.} \ t \geq T.
	\end{align*}
Taking into account (\ref{diffarc}), we are able to conclude
	\begin{align}\label{convergenceinequality}
		\dot{\sigma}(t) \leq - C \sigma(t)^{\frac{\theta}{1-\theta}} \quad \ \mbox{for a.e.} \ t \geq T,
	\end{align}
	where we denote
	$$C:=\frac{M_1^{\frac{\theta}{1-\theta}}}{c^{\frac{1}{1-\theta}} M_2} >0.$$
	From (\ref{convergenceinequality}) we now can deduce the speed of convergence of the trajectory. To this end we distinguish between different values for the \L ojasiewicz exponent $\theta$. 
	\begin{itemize}
		\item If $\theta =\frac{1}{2}$, then (\ref{convergenceinequality}) simplifies to 
		$$\dot{\sigma}(t) \leq - C \sigma(t),$$
		from which exponential convergence
		$$\|z(t)-\bar{z}\| \leq \sigma(t) \leq \alpha e^{-\beta t} \quad \text{for a.e.} \ t \geq T,$$
		with $\alpha = \sigma(T)$ and $\beta=C$, follows by using Gronwall's inequality.
		\item If $\theta \in \left[0,\frac{1}{2}\right)$, then (\ref{convergenceinequality}) yields
		$$\frac{d}{dt} \sigma (t)^{\frac{1-2\theta}{1-\theta}} = \frac{1-2\theta}{1-\theta} \sigma (t)^{-\frac{\theta}{1-\theta}} \dot{\sigma}(t) \leq - C \frac{1-2\theta}{1-\theta} \quad \ \mbox{for a.e.} \ t \geq T.$$
After integration we obtain
		$$\sigma (t)^{\frac{1-2\theta}{1-\theta}} \leq -c_1 t +c_2 \ \mbox{for all} \ t \geq T,$$
for some constants $c_1 , c_2 >0$. This means that there has to exists some $\overline{T}>0$ such that $\sigma(t) = 0$ for all $t \geq \overline{T}$, which implies that $z(t)$ has to be constant on $[\overline T, +\infty)$.
		\item Last, we investigate the case $\theta \in \left(\frac{1}{2}, 1\right)$. Since $\frac{1-2\theta}{1-\theta}<0$ we obtain similarly to the calculation above 
		$$\frac{d}{dt} \sigma (t)^{\frac{1-2\theta}{1-\theta}} \geq C \frac{2\theta - 1}{1-\theta} \quad \ \mbox{for a.e.} \ t \geq T.$$
		Again, by integrating the inequality above we obtain for some constants $\gamma, \delta \in \R$
$$\sigma(t) \leq (\gamma t + \delta)^{\frac{1-2\theta}{1-\theta}} \ \mbox{for all} \ t \geq T.$$
\end{itemize}
\end{proof}

\section{Numerical experiments}
In this last section we aim to illustrate the analytical results with some numerical simulations. Especially, we want to focus on the effect the parameters $\lambda$ as well the stepsizes $\gamma_1$ and $\gamma_2$ have on the asymptotic behavior of the solution trajectories.
As in Section \ref{sec3}, we consider the case $\mu = 1$, where the dynamical system can be written in the explicit form (\ref{dynpalmmu1}). The simulations were carried out in \textsc{Matlab}, where we used the function ode15s.

\begin{figure}[h]
	\centering
	\includegraphics[width=10cm]{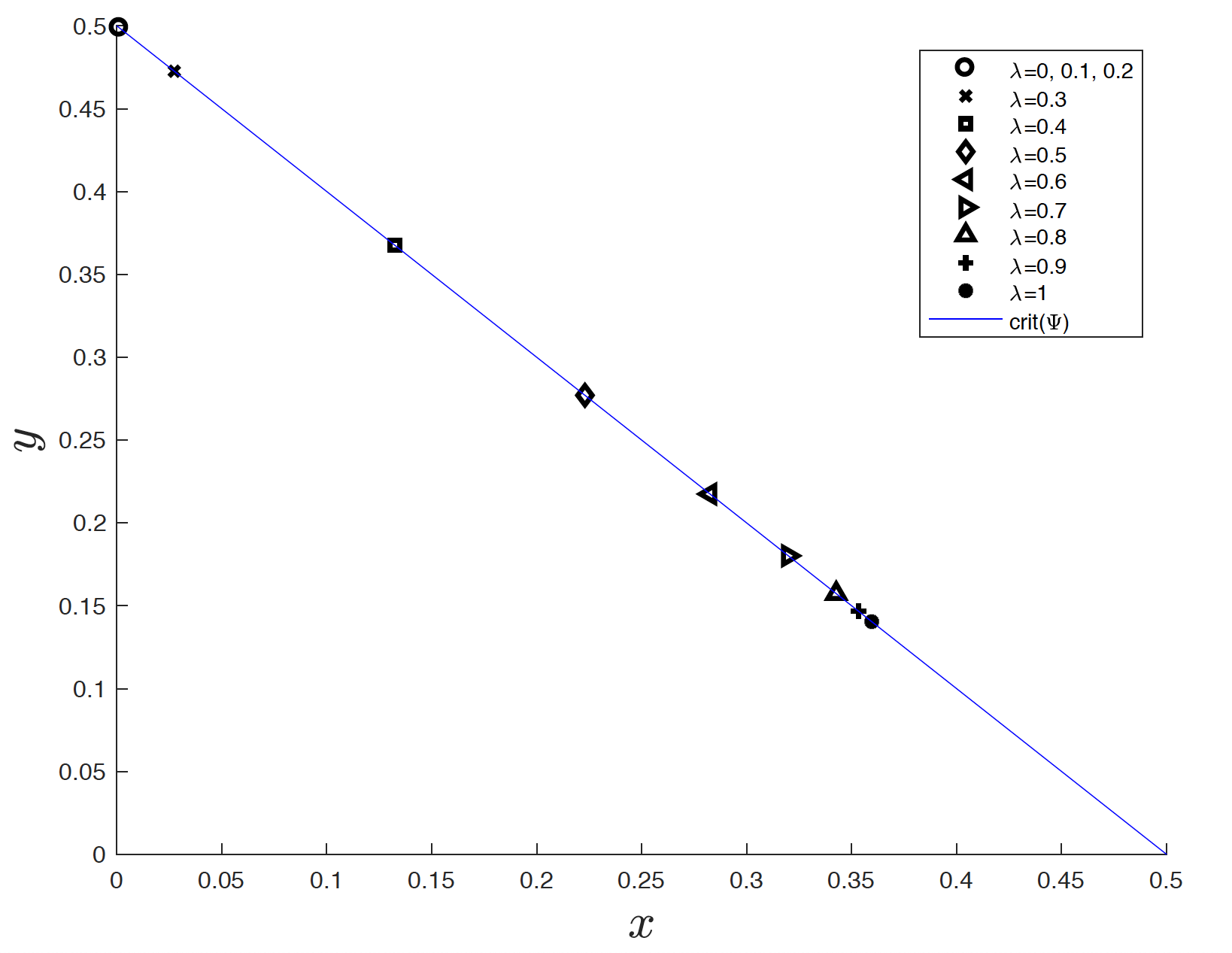}
	\caption{For $c_1=c_2=1$ and $(x_0,y_0)=\left(1,\frac{1}{2}\right)$ the limits of the trajectory change according to the choice of $\lambda$. While for $\lambda \in [0,\:0.3)$ (rather implicit regime), the solution converges to the critical point $(0,\frac{1}{2})$, for $\lambda \in [0.3,\: 1]$ the trajectory approaches other points further down the line of $\crit(\Psi)$. The points are plotted after 100 time steps.}
	\label{limits}
\end{figure}

\begin{example}
	In the first experiment, we investigated the optimization problem
	\begin{equation}\label{experiment1}
		\min_{(x,y) \in \R \times \R} \Psi(x,y) := |x|+|y|+(1-x-y)^2,
	\end{equation}
which corresponds to problem (\ref{p}) in the special setting $f,\:g: \R \to \R$,  $f=g=|\cdot|$, and $H: \R \times \R \to \R$, $H(x,y)=(1-x-y)^2$. The critical points of $\Psi$ are given by the set $$\crit(\Psi) = \left\{(\bar{x},\bar{y}) \in \R_+ \times \R_+: \bar{x}+\bar{y}=\frac{1}{2} \right\},$$
where all elements of this set are optimal solutions for (\ref{experiment1}). We chose $(x_0,y_0)=\left(1,\frac{1}{2}\right)$ as starting point. 

\begin{figure}[h]
	\begin{subfigure}{}
	\includegraphics[width=8.2cm]{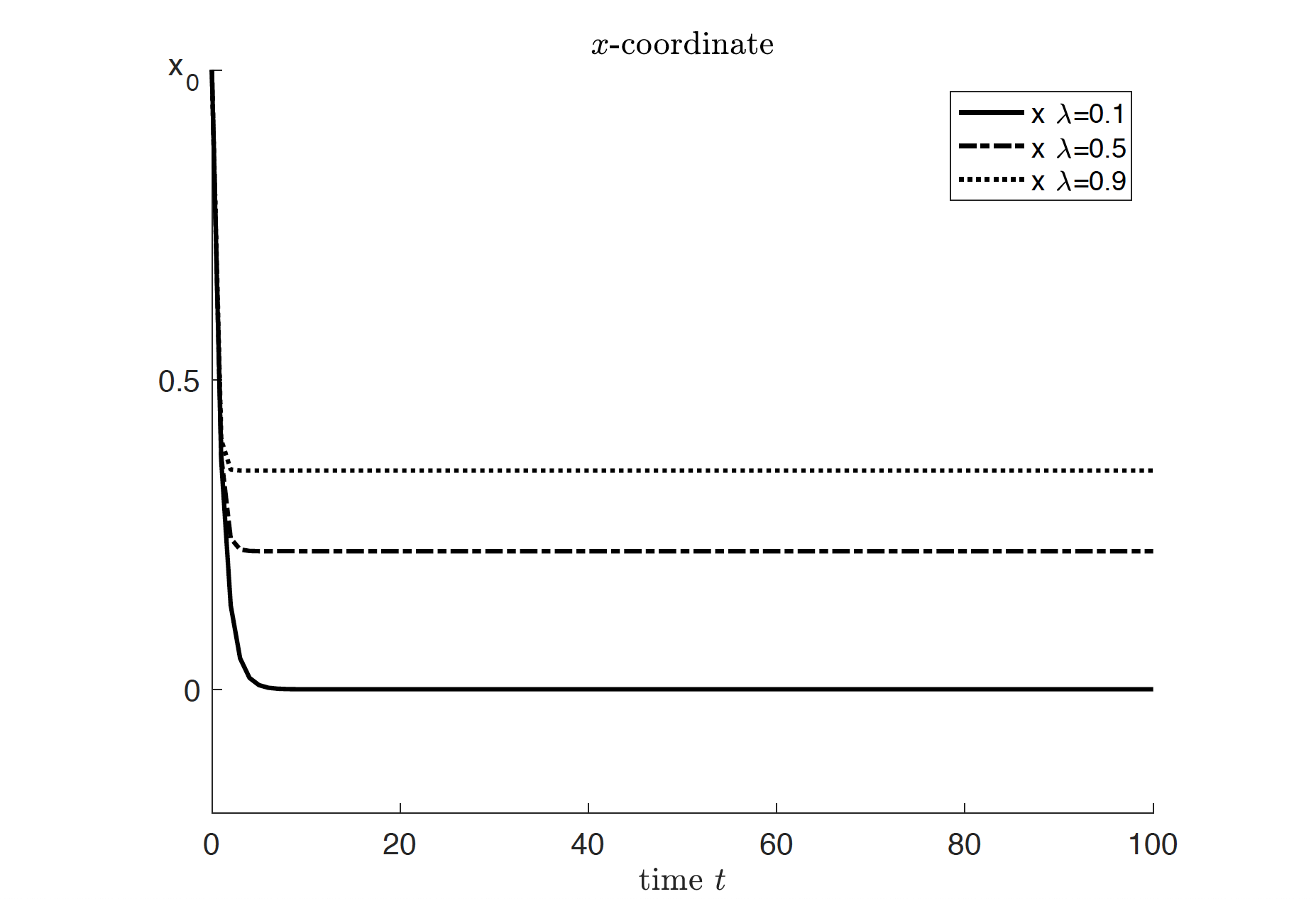}
	\end{subfigure}
	\hspace{0.2em} 
	\begin{subfigure}{}
		\includegraphics[width=8.2cm]{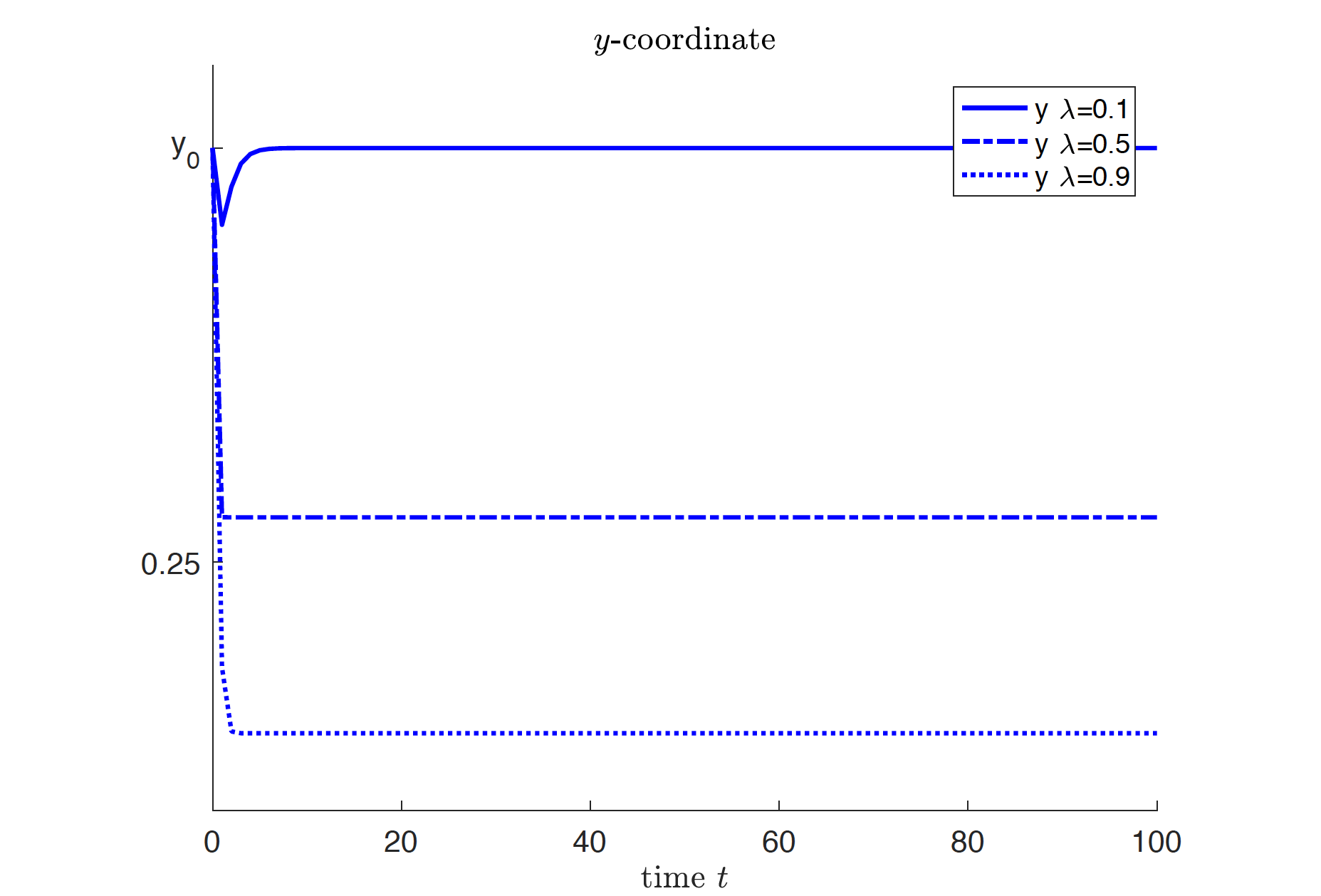}
	\end{subfigure}
	\caption{For $c_1=c_2=1$ and $(x_0,y_0)=\left(1,\frac{1}{2}\right)$,  three different solution trajectories approaching three different critical values are plotted. \emph{On the left:} for $\lambda=0.1$, the $x$-coordinate of the solution $(x(t),y(t))$ converges to $\bar{x}=0$ (black solid), for $\lambda=0.5$, to $\bar{x}=0.223$ (black dashed) and, for $\lambda=0.9$, to $\bar{x}=0.3533$ (black dotted). \emph{On the right:} for for $\lambda=0.1$,  $y(t)$ converges to $\bar{y}=0.5$ (blue solid), for $\lambda=0.5$, to $\bar{y}=0.277$,  (blue dashed) and, for $\lambda=0.9$, to $\bar{y}=0.1467$ (blue dotted).}
	\label{plot1}
\end{figure}

\begin{figure}[H]
	\centering
	\begin{subfigure}{} 
		\includegraphics[width=8.2cm]{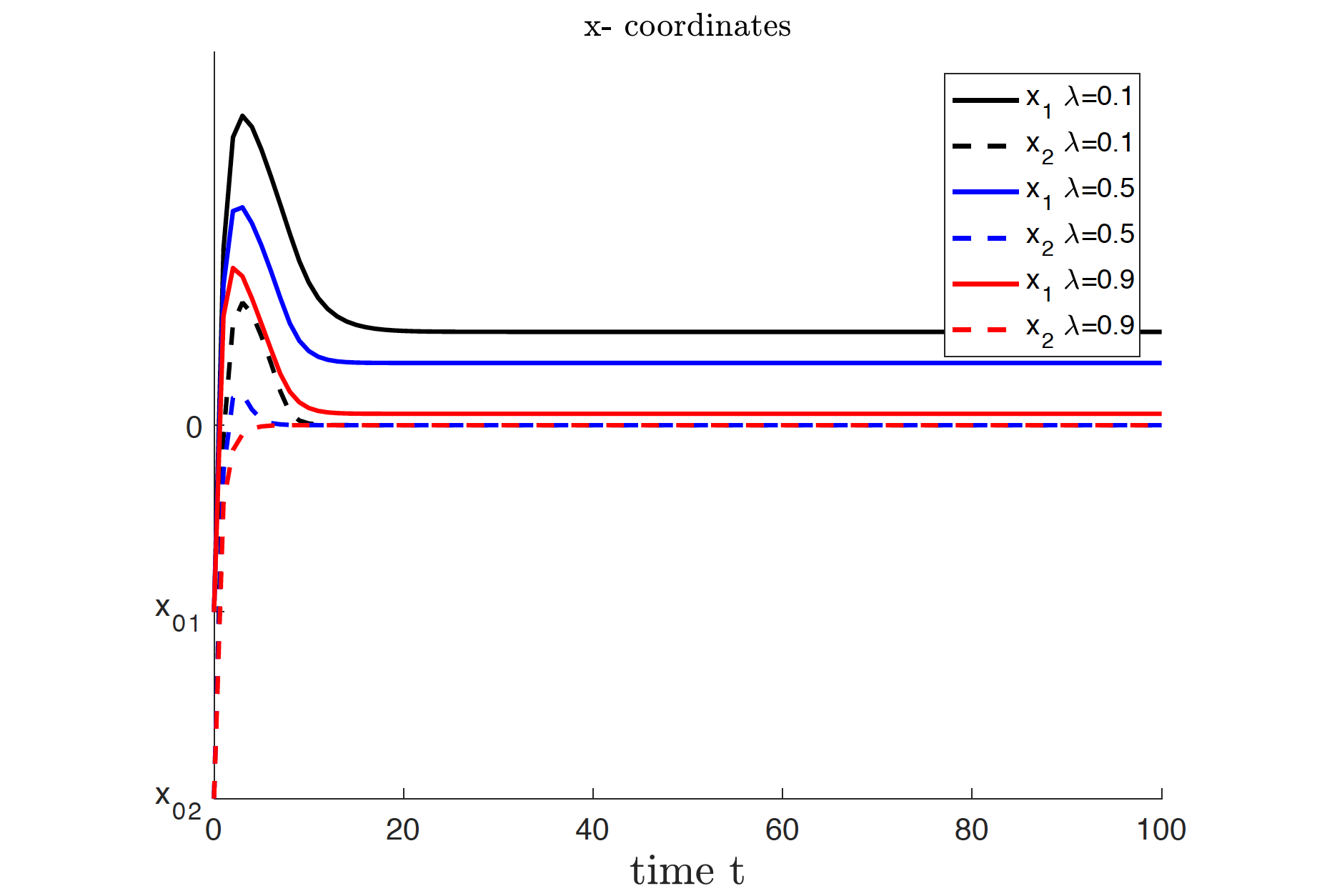}	
	\end{subfigure}
	\hspace{0.1em} 
	\begin{subfigure}{}
		\includegraphics[width=8.2cm]{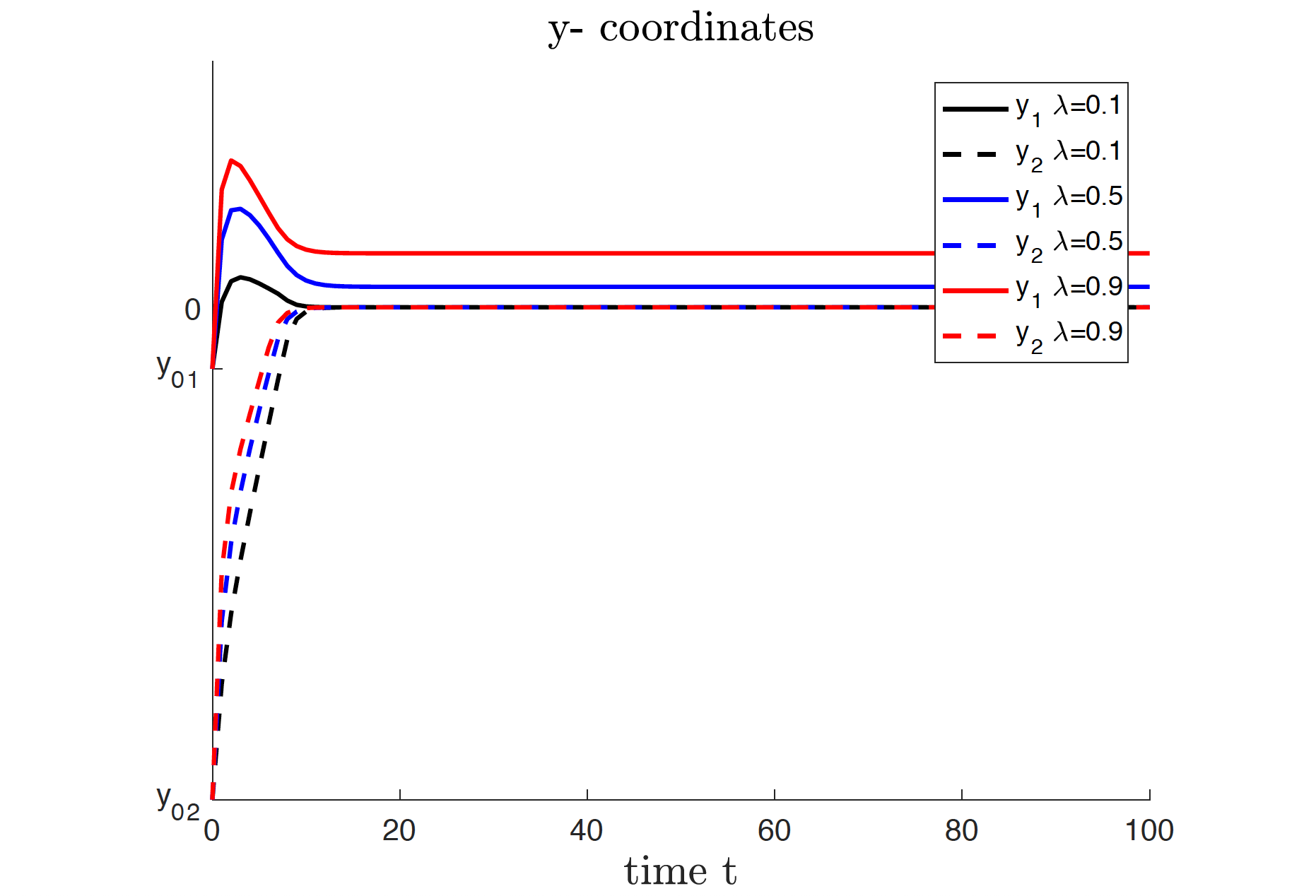}
	\end{subfigure}
	\caption{\emph{On the left:} $x(t)\in \R^2$ converges to $\bar{x}=(0.5,\:0)$, for $\lambda=0.1$ (black), to $\bar{x}=(0.3333,\:0)$, for $\lambda=0.5$ (blue), and to $\textbf{}(0.0612,\:0)$, for $\lambda=0.9$ (red). \emph{On the right:} $y(t)\in \R^2$ converges to $\bar{y}=(0,\:0)$, for $\lambda=0.1$ (black), to $\bar{y}=(0.1666,\:0)$, for $\lambda=0.5$ (blue), and to $\bar{y}=(0.4388,\:0)$ for $\lambda=0.9$ (red).}
	\label{plots3}
\end{figure}

The first observation was that the rate of implicity, which is given by the the value of $\lambda$ in the second component of the system (\ref{dynpalmmu1}), plays a crucial role in determining to which critical point the solution trajectory converges, provided to be in a setting where the stepsizes 
$$c_1:=\gamma_1 L \text{ and } c_2:=\gamma_2 L$$
are chosen such that (\ref{g1}) is fulfilled for all $\lambda \in [0,\:1]$ (see Figure \ref{plot1}). One should also remark that the rate of implicitness seems to influence the stability of convergence also in dependence of the shape of the problem. Here, for higher values of $\lambda$, 
this means for more explicit systems, the trajectory converges faster, while for more implicit systems, which arise when $\lambda$ is close to $0$, the trajectory first moves away from the limit set until the convergence starts, as one can see as a peak in the functions for the $x$-coordinates. For an overview over the limit points the different trajectories converge to we refer to Figure \ref{limits}.

\begin{figure}[H]
	\centering
	\begin{subfigure}{} 
		\includegraphics[width=8cm]{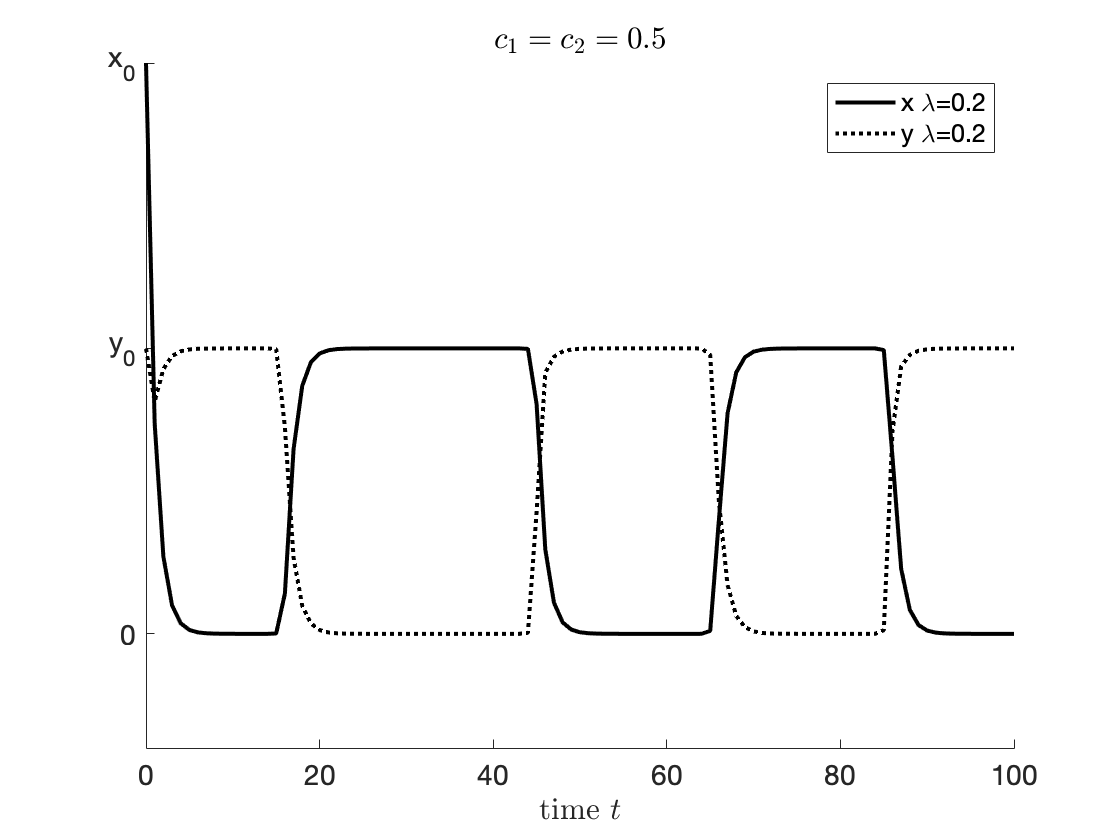}	
	\end{subfigure}
	\hspace{0.5em} 
	\begin{subfigure}{}
		\includegraphics[width=8cm]{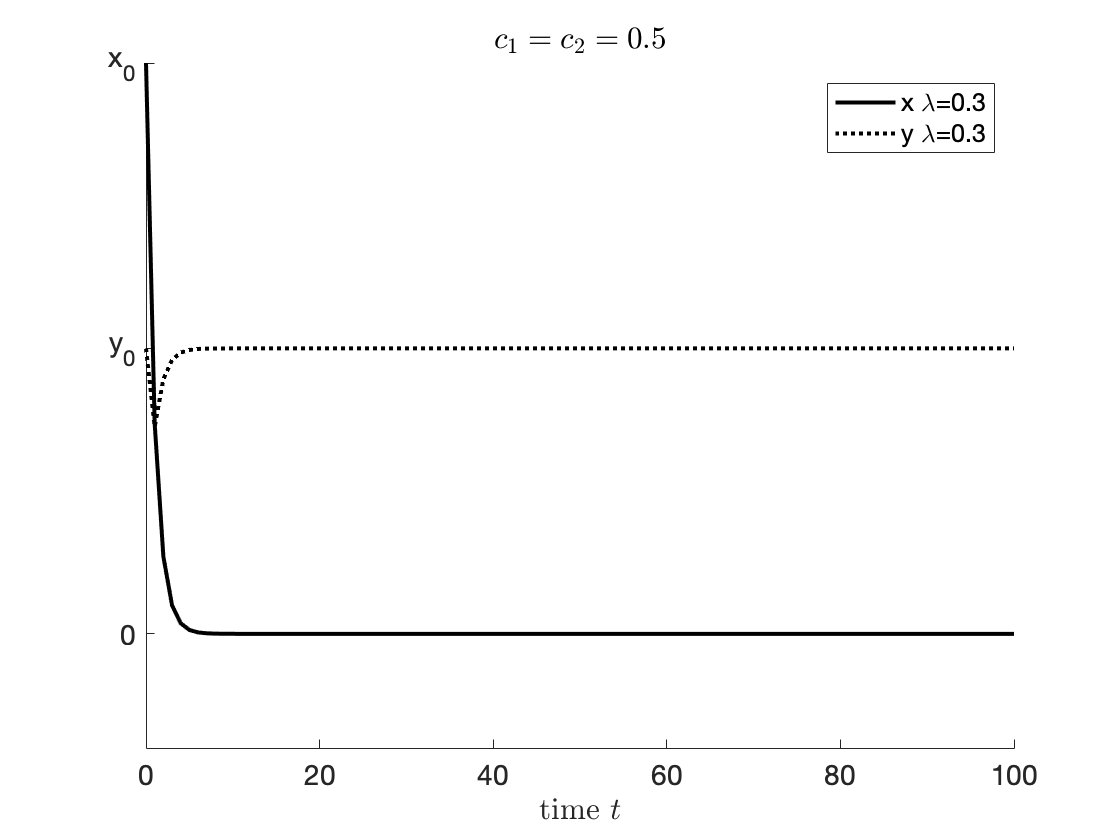}
	\end{subfigure}
	\\
	\begin{subfigure}{}
		\includegraphics[width=8cm]{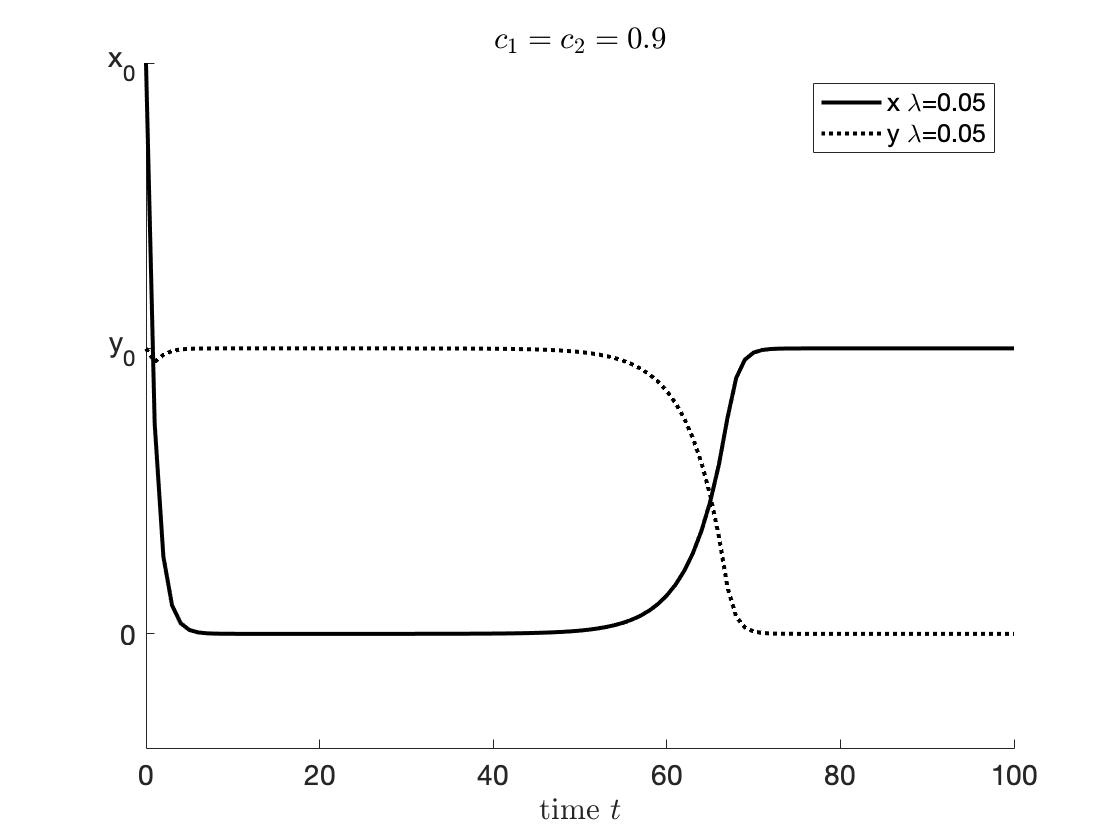}	
	\end{subfigure}
	\hspace{0.5em} 
	\begin{subfigure}{}
		\includegraphics[width=8cm]{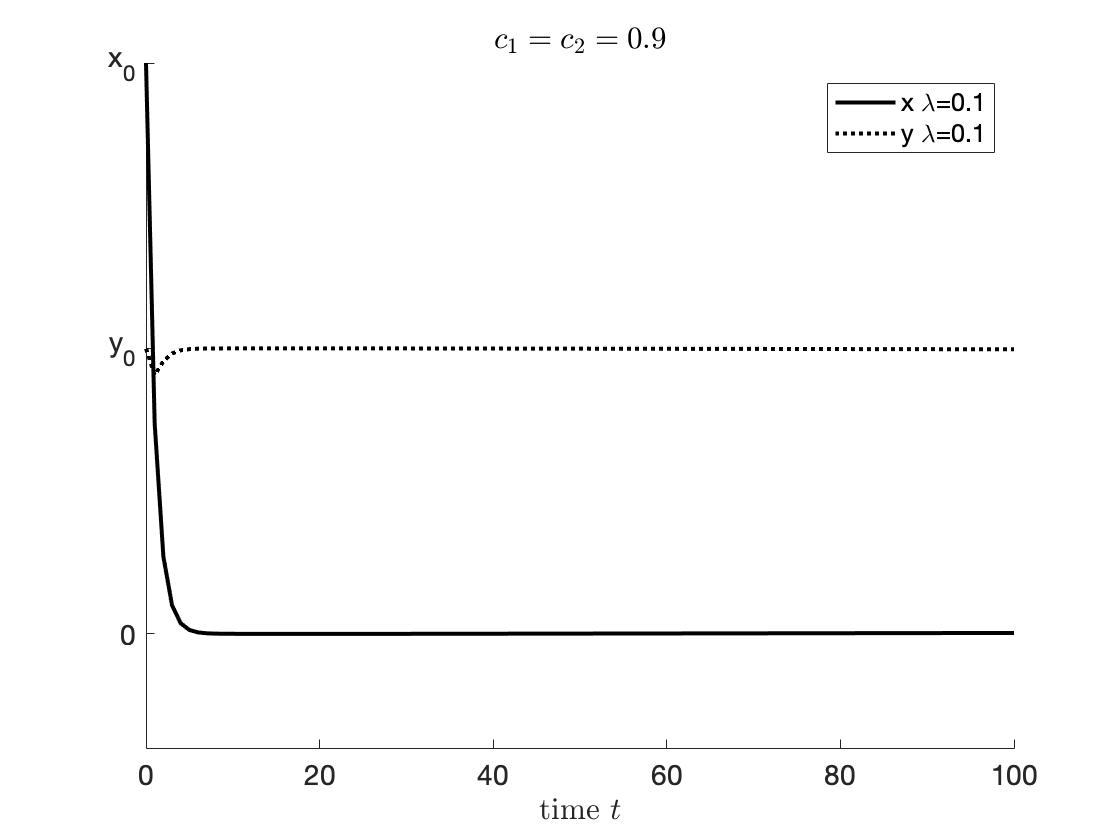}
	\end{subfigure}
	\caption{\emph{Top row:} on the left, $(x(t),y(t))$ moves periodically between the critical points $(0,\:0.5)$ and $(0.5,\:0)$ for $c_1=c_2=0.5$ and $\lambda = 0.2$.  On the right one can see the trajectory converging for $c_1=c_2=0.5$ and $\lambda=0.3$. \emph{Bottom row:} on the left, $(x(t),y(t))$ shows periodic behavior for $c_1=c_2=0.9$ and $\lambda = 0.05$, while on the right, for $\lambda=0.1$, the trajectory converges.}
	\label{plots2}
\end{figure}

In addition to the above observations we further investigated the problem (\ref{experiment1}) in higher dimensions, namely
\begin{equation}\label{experiment2}
\min_{(x,y) \in \R^2 \times \R^2} \Psi(x,y) := \|x\|_1+\|y\|_1+(1-x_1-x_2-y_1-y_2)^2,
\end{equation}
in order to emphasize the different convergence behavior of the solution trajectories with respect to the parameter $\lambda$. Here,$f,\:g: \R^2 \to \R$ are both given by the $1$-norm  and $H: \R^4 \to \R$ is of the form $H(x,y)=(1-x_1-x_2-y_1-y_2)^2$, where $x=(x_1,\:x_2)$, $y=(y_1,\:y_2)$. The set of critical points of $\Psi$ are given by the set $$\crit(\Psi) = \left\{(\bar{x},\bar{y}) \in \R^2_+ \times \R^2_+: \bar x_1 +\bar x_2 +\bar y_1 +\bar y_2 =\frac{1}{2} \right\},$$
where all elements of this set are optimal solutions for (\ref{experiment2}). We chose as starting value $(x_0,y_0)=\left(-1,\:-2,\:-\frac{1}{2},-4\right)$. 
Again, we were able to observe that for $c_1=c_2=5$, which is a setting in which (\ref{g1}) is fulfilled for all $\lambda \in [0,\:1]$, the limit points of the trajectory are determined by the choice of $\lambda$. This behaviour can be seen in Figure \ref{plots3}. Once more one can notice that also here the more explicit the system is (i.e. the closer $\lambda$ is to 1), the more stable the trajectory converges to a limit point. 

Moreover, we could observe that the validity of inequality (\ref{g1}), which in our special setting is of the form
$$ \min\{c_1 L, c_2L\} > \left(6+\frac{4}{c_1^2L^2}+\frac{4+24(1-\lambda)^2}{c_2^2L^2}+\frac{16(1-\lambda)^2}{c_1^2c_2^2L^4}\right)^{\frac{1}{2}}(1+\lambda+\lambda^2)+\frac{\lambda}{2},$$
is essential for the solution trajectory to converge. The plots of Figure \ref{plots2} show that the smaller $\lambda$ is chosen, the bigger the stepsizes have to be: $c_1=c_2=0.5$ just provide convergence for $\lambda \in [0.3,\:1]$, while for the very implicit system, i.e. for $\lambda \in [0,\:0.3)$ it fails. Bigger stepsizes $c_1=c_2=0.9$ assure convergence for a wider range of implicity, i.e. for $\lambda \in [0.1,\:1]$, while choosing $c_1=c_2=1$ guarantees convergence for all $\lambda \in [0,\:1]$.
\end{example}

\begin{example}
	In the second numerical experiment we considered the minimization problem 
	\begin{equation*}
	\min_{(x,y) \in \R \times \R} \Psi(x,y) := |x|+L_2(y)-\frac{1}{5}(1-x-y)^2,
	\end{equation*}
	where, for $\alpha > 0$,
	$$L_\alpha : \R \rightarrow \R, \ L_{\alpha}(y):=\begin{cases} \frac{y^2}{2}, \quad &|y|\leq \alpha \\ \alpha|y|-\frac{\alpha^2}{2} \quad &\text{otherwise}, \end{cases}$$
defines the \emph{Huber function}. Taking $f: \R \to \R$, $f(x):=|x|$, and $g: \R \to \R$, $g(y):=L_2(y)$, as well as $H: \R \times \R \to \R$, $H(x,y):=(1-x-y)^2$, we are again in the setting of problem (\ref{dynpalmmu1}). The set of critical points of $\Psi$ is given by
	$$\crit(\Psi) = \left\{\left(-\frac{1}{2},\:1\right); \left(0,\:-\frac{2}{3}\right); \left(\frac{5}{2},\:1\right)\right\},$$
while $\left(0,\: -\frac{2}{3}\right)$ is the optimal solution.
	
	\begin{figure}[H]
		\centering
		\begin{subfigure}{} 
			\includegraphics[width=8.1cm]{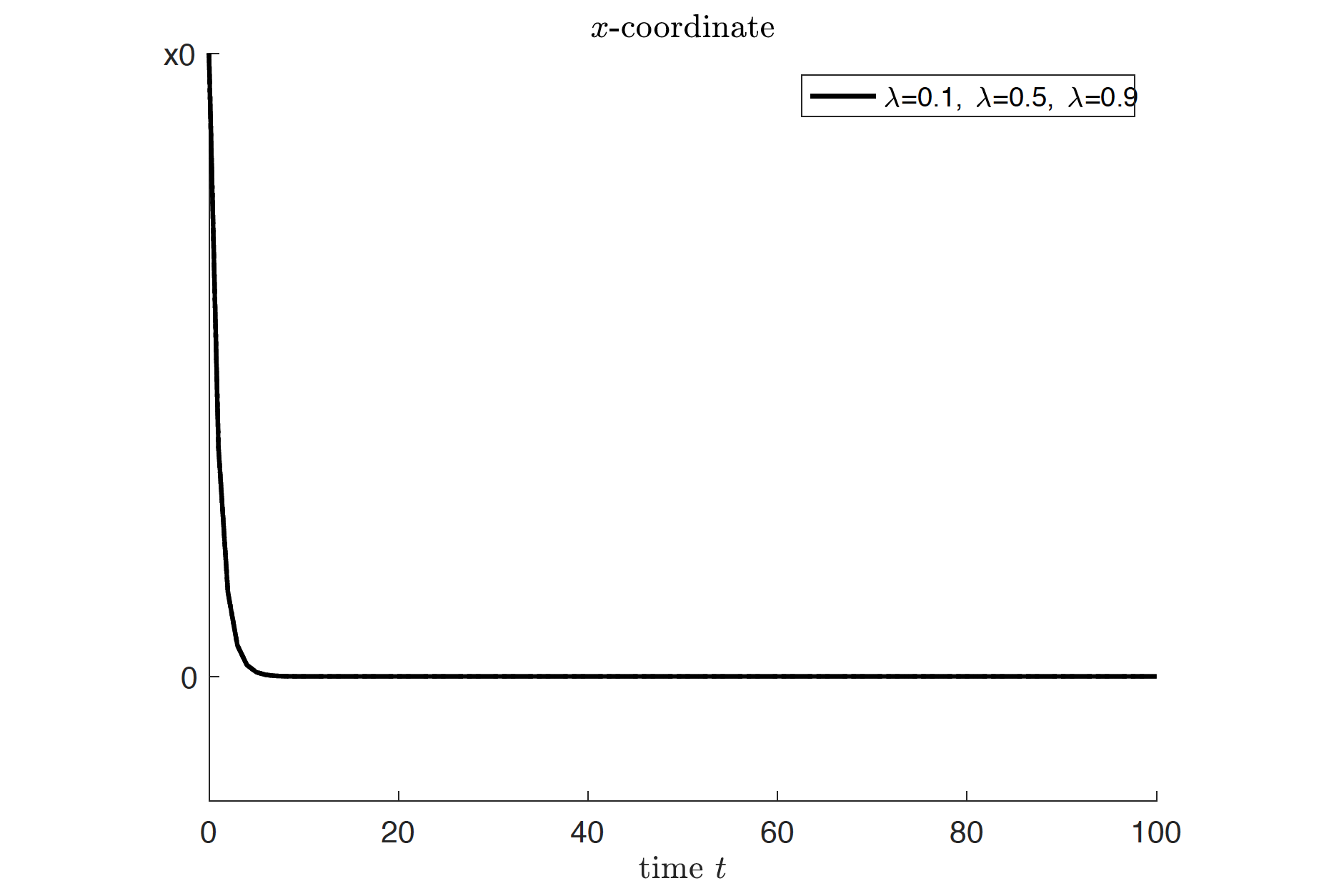}	
		\end{subfigure}
		\hspace{0.2em} 
		\begin{subfigure}{}
			\includegraphics[width=8.1cm]{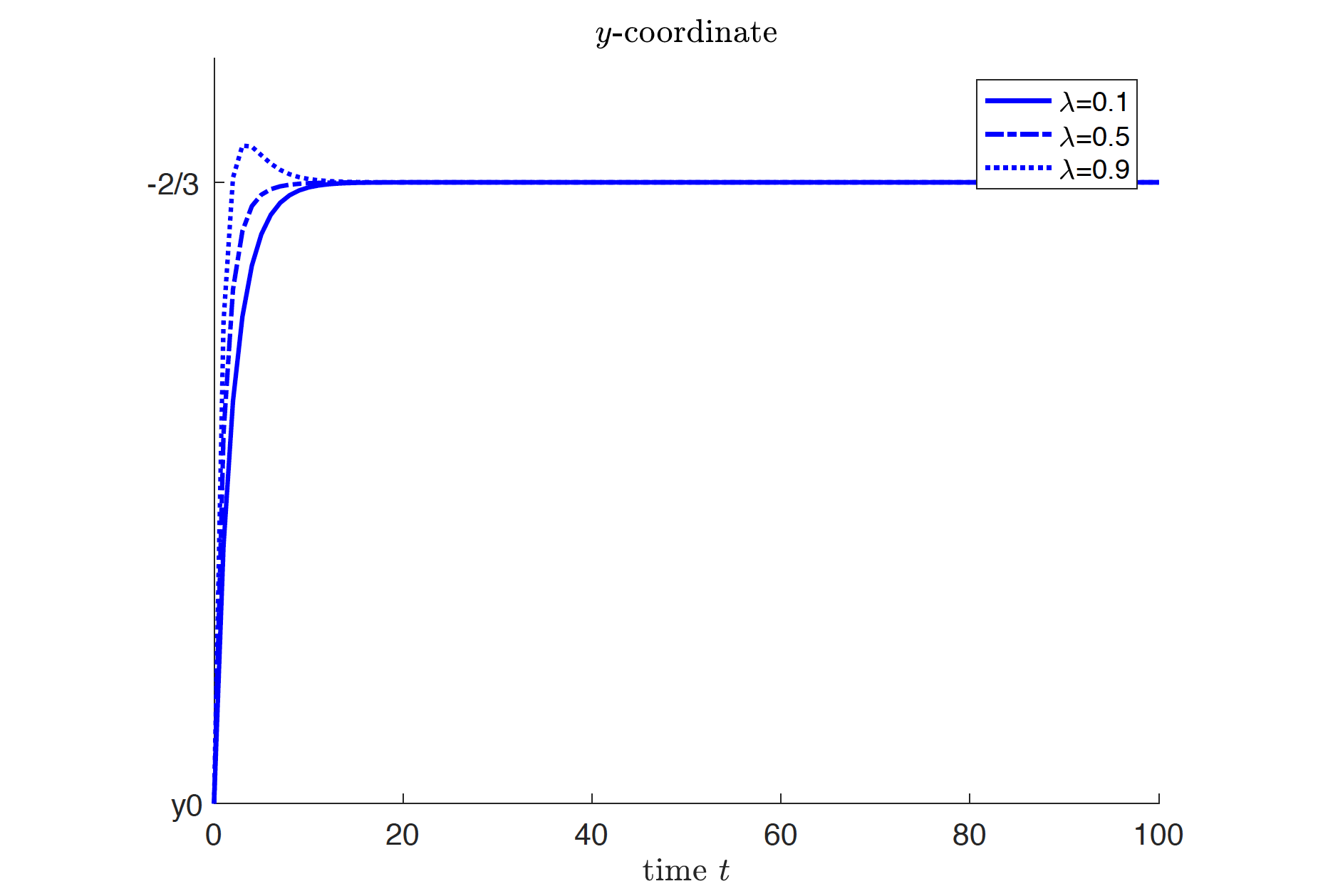}
		\end{subfigure}
		\caption{\emph{On the left:} $x(t)$ converges to $\bar{x}=0$, for $\lambda=0.1$, $\lambda=0.5$, $\lambda=0.9$. \emph{On the right:} $y(t)$ converges to $\bar{y}=-\frac{2}{3}$ with different convergence behavior, for $\lambda=0.1$, $\lambda=0.5$, $\lambda=0.9$}
		\label{plots4}
	\end{figure}
	In Figure \ref{plots4} again the $x$- and $y$-solution trajectories for three dynamical systems with different rates of implicitness are plotted. As starting value $(x_0,\:y_0)=(1,\:-1)$ was chosen. The Lipschitz constant of $H$ allows to choose as stepsizes $c_1=c_2=0.3$ in order to verify 
	\eqref{g1}. One can observe that the values of $\lambda$ do not affect the convergence of the $x$-trajectories at all, while for the $y$-trajectories the implicitness has influence on the stability of convergence. For the relatively explicit system with $\lambda=0.9$ one can see that the solution 
	first moves away from the limit point, before the convergence starts. The trajectory corresponding to $\lambda=0.5$ reaches the limit point the fastest.  
\end{example}


\begin{thebibliography}{9}
	   \bibitem{abb1} B. Abbas, H. Attouch, \emph{Dynamical systems and forward-backward algorithms associated with the sum of a convex subdifferential and a monotone cocoercive operator}, Optimization 64(10), 2223--2252 (2015)
		\bibitem{abb} B. Abbas, H. Attouch, B. Svaiter, \emph{Newton-like dynamics and forward-backward methods for structured monotone inclusions in Hilbert spaces}, Journal of Optimization Theory and Applications 161, 331--360 (2014)
		\bibitem{al} F. Alvarez, H. Attouch, J. Bolte, P. Redont, \emph{A second-order gradient-like dissipative dynamical system with Hessian-driven damping. Application to optimization and mechanics}, Journal de Mathématiques Pures et Appliquées (9) 81(8), 747--779 (2002)
		\bibitem{ab} H. Attouch, J. Bolte. \emph{On the convergence of the proximal algorithm for nonsmooth functions involving analytic features}, Mathematical Programming 116(1) Series B, 5--16 (2007)
		\bibitem{ant} A.S. Antipin, \emph{Minimization of convex functions on convex sets by means of differential equations}, (Russian) Differentsial'nye Uravneniya 30(9), 1475--1486 (1994)
		\bibitem{att3} H. Attouch, J. Bolte, B. F. Svaiter, \emph{Convergence of descent methods for semi-algebraic and tame problems: proximal algorithms, forward-backward splitting, and regularized Gauss-Seidel methods}, Mathematical Programming 137(1-2) Series A, 91–-129 (2013)
		\bibitem{att} H. Attouch, G. Buttazzo and G. Michaille, \emph{Variational Analysis in Sobolev and BV Spaces: Applications to PDEs and Optimization}, MAOS-SIAM Series on Optimization, Philadelphia (2014)
		\bibitem{att2} H. Attouch, B. F. Svaiter, \emph{A continuous dynamical Newton-like approach to solving monotone inclusions}, SIAM Journal on Control  and Optimization 49, 574--598 (2011)
		\bibitem{bai} J.B. Baillon, H. Brézis, \emph{Une remarque sur le comportement asymptotique des semigroupes non linéares}, Houston Journal of Mathematics 2(1), 5--7 (1976)
		\bibitem{ban} S. Banert, R.I. Bo\c{t}, \emph{A forward-backward-forward differential equation and its asymptotic properties}, Journal of Convex Analysis 25(2), 371--388 (2018)
		\bibitem{bc} H.H. Bauschke, P.L. Combettes, \emph{Convex Analysis and Monotone Operator Theory in Hilbert Spaces}, CMS Books in Mathematics, Springer (2017)
		\bibitem{bol1} J. Bolte, \emph{Continuous gradient projection method in Hilbert spaces}, Journal of Optimization Theory and its Applications 119(2), 235--259 (2003)
		\bibitem{bol2} J. Bolte, A. Daniilidis, A. Lewis, \emph{The  \L ojasiewicz inequality for nonsmooth subanalytic functions with applications to subgradient dynamical systems}, SIAM Journal on Optimization 17(4), 1205–-1223 (2006)
		\bibitem{bol} J. Bolte, S. Sabach, M. Teboulle, \emph{Proximal alternating linearized minimization for nonconvex and nonsmooth problems}, Mathematical Programming 146 (1-2) Series A, 459..494 (2014) 
		\bibitem{bot1} R.I. Bo\c{t}, E.R. Csetnek, \emph{A dynamical system associated with the fixed points set of a nonexpansive operator}, Journal of Dynamics and Differential Equations 29(1), 155--168  (2017)
		\bibitem{bot} R.I. Bo\c{t}, E.R. Csetnek, \emph{A forward-backward dynamical approach to the minimization of the sum of a nonsmooth convex with a smooth nonconvex function}, ESAIM: Control, Optimisation and Calculus of Variations 24(2), 463--477 (2018)
		\bibitem{bot3}  R.I. Bo\c{t}, E.R. Csetnek, S.C. L\'aszl\'o, \emph{An inertial forward-backward algorithm for the minimization of the sum of two nonconvex functions}, EURO Journal on Computational Optimization 4(1), 3--25 (2016) 
		\bibitem{bot4}  R.I. Bo\c{t}, E.R. Csetnek, S.C. L\'aszl\'o, \emph{Approaching nonsmooth nonconvex minimization through second order proximal-gradient dynamical systems}, Journal of Evolution Equations 18(3), 1291--1318 (2018) 
		\bibitem{bot2}  R.I. Bo\c{t}, E.R. Csetnek, S.C. L\'aszl\'o, \emph{A primal-dual dynamical approach to structured convez minimization problems}, arXiv:1905.08290 [math.OC] (2019)
		\bibitem{brez} H. Brézis, \emph{Operateurs Maximaux Monotones: Et Semi-Groupes De Contractions Dans Les Espaces De Hilbert}, North-Holland Mathematics Studies 5, North Holland (1973)
		\bibitem{bru} R.E. Bruck, Jr., \emph{Asymptotic convergence of nonlinear contraction semigroups in Hilbert space}, Journal of Functional Analysis 18, 15--26 (1975)
		\bibitem{cran} M.G. Crandall, A. Pazy, \emph{Semi-groups of nonlinear contractions and dissipative sets}, Journal of Functional Analysis 3, 376--418 (1969)
		\bibitem{cse} E.R. Csetnek, Y. Malitsky, M. K. Tam, \emph{Shadow Douglas-Rachford Splitting for Monotone Inclusions}, Applied Mathematics and Optimization 80, 665--678 (2019)
		\bibitem{dav} D. Davis, D. Drusvyatskiy, S. Kakade, J. D. Lee, \emph{Stochastic subgradient method converges on tame functions}, Foundations of Computational Mathematics (2019)
		\bibitem{haraux} A. Haraux, \emph{Systémes Dynamiques Dissipatifs et Applications}, Recherches en Mathématiques Appliquées 17, Masson, Paris (1991)
		\bibitem{haraux2} A. Haraux, M. Jendoubi, \emph{Convergence of solutions of second-order gradient-like systems with analytic nonlinearities}, Journal of Differential Equations 144(2), 313--320  (1998)
\bibitem{lipong} G. Li, T.K. Pong, \emph{Calculus of the exponent of the Kurdyka-\L ojasiewicz inequality and its applications to linear convergence of first-order methods}, Foundations of Computational Mathematics 18(5), 1199-1232 (2018)
		\bibitem{loj} S.  Lojasiewicz, \emph{Une Propriété Topologique des Sous-Ensembles Analytiques Réels}, Les Équations aux Dérivées Partielles, Éditions du Centre National de la Recherche Scientifique Paris, 87--89 (1963)
		\bibitem{mor} B. Mordukhovich, \emph{Variational Analysis and Generalized Differentiation}, Springer, Berlin (2006)
		\bibitem{rocky} R.T. Rockafellar, R. Wets, \emph{Variational Analysis}, Springer, Berlin (1998)
		\bibitem{sim} L. Simon, \emph{Asymptotics for a class of nonlinear evolution equations, with applications to geometric problems}, Annals of Mathematics (2)118, 525--571 (1983)
		\bibitem{teschl} G. Teschl, \emph{Ordinary Differential Equations and Dynamical Systems}, Graduate Studies in Mathematics 140, American Mathematical Society (2012)
\end{thebibliography}
\end{document}